\documentclass{amsart}

\usepackage{amssymb, amsmath, amsthm, stmaryrd, amscd, mathrsfs}

\usepackage{verbatim}
\usepackage{a4wide}

\usepackage{enumerate}

\allowdisplaybreaks

\usepackage{hyperref}

\usepackage{color}

\theoremstyle{plain}
\newtheorem{theorem}{Theorem}[section]
\theoremstyle{remark}
\newtheorem{remark}[theorem]{Remark}

\theoremstyle{plain}
\newtheorem{corollary}[theorem]{Corollary}
\newtheorem{lemma}[theorem]{Lemma}
\newtheorem{proposition}[theorem]{Proposition}

\numberwithin{equation}{section}

\def\N{{\mathbb N}}

\def\R{{\mathbb R}}

\def\C{{\mathbb C}}

\newcommand{\E}{{\mathbb E}}

\newcommand{\F}{{\mathscr F}}

\newcommand{\g}{\gamma}

\newcommand{\e}{\varepsilon}

\newcommand{\tr}{\text{tr}}
\newcommand{\ext}{\text{ext}}

\newcommand{\calL}{{\mathscr L}}

\newcommand{\one}{{{\bf 1}}}

\newcommand{\wh}{\widehat}

\newcommand{\nn}{|\!|\!|}

\newcommand{\Schw}{{\mathscr S}}
\newcommand{\TD}{{\mathscr S'}}

\def\typeout#1{\message{^^J}\message{#1}\message{^^J}}
\newif\ifSRCOK \SRCOKtrue
\newcount\PAGETOP
\newcount\LASTLINE
\global\PAGETOP=1
\global\LASTLINE=-1
\def\EJECT{\SRC\eject}
\def\WinEdt#1{\typeout{:#1}}
\gdef\MainFile{\jobname.tex}
\gdef\CurrentInput{\MainFile}
\newcount\INPSP
\global\INPSP=0
\def\SRC{\ifSRCOK%
  \ifnum\inputlineno>\LASTLINE%
    \ifnum\LASTLINE<0%
      \global\PAGETOP=\inputlineno%
    \fi%
    \global\LASTLINE=\inputlineno%
    \ifnum\INPSP=0%
      \ifnum\inputlineno>\PAGETOP%
        
      \fi%
    \else%
      
    \fi%
  \fi%
\fi}
\def\PUSH#1{%
\SRC%
\ifnum\INPSP=0 \global\let\INPSTACKA=\CurrentInput \else%
\ifnum\INPSP=1 \global\let\INPSTACKB=\CurrentInput \else%
\ifnum\INPSP=2 \global\let\INPSTACKC=\CurrentInput \else%
\ifnum\INPSP=3 \global\let\INPSTACKD=\CurrentInput \else%
\ifnum\INPSP=4 \global\let\INPSTACKE=\CurrentInput \else%
\ifnum\INPSP=5 \global\let\INPSTACKF=\CurrentInput \else%
               \global\let\INPSTACKX=\CurrentInput \fi\fi\fi\fi\fi\fi%
\gdef\CurrentInput{#1}%
\WinEdt{<+ \CurrentInput}%
\global\LASTLINE=0%
\ifSRCOK\fi%
\global\advance\INPSP by 1}
\def\POP{%
\ifnum\INPSP>0 \global\advance\INPSP by -1  \fi%
\ifnum\INPSP=0 \global\let\CurrentInput=\INPSTACKA \else%
\ifnum\INPSP=1 \global\let\CurrentInput=\INPSTACKB \else%
\ifnum\INPSP=2 \global\let\CurrentInput=\INPSTACKC \else%
\ifnum\INPSP=3 \global\let\CurrentInput=\INPSTACKD \else%
\ifnum\INPSP=4 \global\let\CurrentInput=\INPSTACKE \else%
\ifnum\INPSP=5 \global\let\CurrentInput=\INPSTACKF \else%
               \global\let\CurrentInput=\INPSTACKX \fi\fi\fi\fi\fi\fi%
\WinEdt{<-}%
\global\LASTLINE=\inputlineno%
\global\advance\LASTLINE by -1%
\SRC}
\def\INPUT#1{\relax}

\let\originalxxxeverypar\everypar
\newtoks\everypar
\originalxxxeverypar{\the\everypar\expandafter\SRC}
\everymath\expandafter{\the\everymath\expandafter\SRC}
\output\expandafter{\expandafter\SRCOKfalse\the\output}

\newif\ifSRCOK \SRCOKtrue
\newcount\PAGETOP
\newcount\LASTLINE
\global\PAGETOP=1
\global\LASTLINE=-1
\gdef\MainFile{\jobname.tex}
\gdef\CurrentInput{\MainFile}
\newcount\INPSP
\global\INPSP=0
\def\EJECT{\SRC\eject}
\def\WinEdt#1{\typeout{:#1}}
\def\SRC{\ifSRCOK%
  \ifnum\inputlineno>\LASTLINE%
    \ifnum\LASTLINE<0%
      \global\PAGETOP=\inputlineno%
    \fi%
    \global\LASTLINE=\inputlineno%
    \ifnum\INPSP=0%
      \ifnum\inputlineno>\PAGETOP%
      \fi%
    \else%
    \fi%
  \fi%
\fi}
\def\PUSH#1{%
\SRC%
\ifnum\INPSP=0 \global\let\INPSTACKA=\CurrentInput \else%
\ifnum\INPSP=1 \global\let\INPSTACKB=\CurrentInput \else%
\ifnum\INPSP=2 \global\let\INPSTACKC=\CurrentInput \else%
\ifnum\INPSP=3 \global\let\INPSTACKD=\CurrentInput \else%
\ifnum\INPSP=4 \global\let\INPSTACKE=\CurrentInput \else%
\ifnum\INPSP=5 \global\let\INPSTACKF=\CurrentInput \else%
               \global\let\INPSTACKX=\CurrentInput \fi\fi\fi\fi\fi\fi%
\gdef\CurrentInput{#1}%
\WinEdt{<+ \CurrentInput}%
\global\LASTLINE=0%
\ifSRCOK\fi%
\global\advance\INPSP by 1}
\def\POP{%
\ifnum\INPSP>0 \global\advance\INPSP by -1  \fi%
\ifnum\INPSP=0 \global\let\CurrentInput=\INPSTACKA \else%
\ifnum\INPSP=1 \global\let\CurrentInput=\INPSTACKB \else%
\ifnum\INPSP=2 \global\let\CurrentInput=\INPSTACKC \else%
\ifnum\INPSP=3 \global\let\CurrentInput=\INPSTACKD \else%
\ifnum\INPSP=4 \global\let\CurrentInput=\INPSTACKE \else%
\ifnum\INPSP=5 \global\let\CurrentInput=\INPSTACKF \else%
               \global\let\CurrentInput=\INPSTACKX \fi\fi\fi\fi\fi\fi%
\WinEdt{<-}%
\global\LASTLINE=\inputlineno%
\global\advance\LASTLINE by -1%
\SRC}
\def\INPUT#1{\relax}
\let\OldINCLUDE=\include
\def\include#1{
\EJECT%
\PUSH{#1.tex}%
\OldINCLUDE{#1}%
\POP}

\let\originalxxxeverypar\everypar
\newtoks\everypar
\originalxxxeverypar{\the\everypar\expandafter\SRC}
\everymath\expandafter{\the\everymath\expandafter\SRC}
\let\zzzxxxbibliography=\bibliography
\def\bibliography#1{\PUSH{\jobname.bbl}\zzzxxxbibliography{#1}\POP}
\output\expandafter{\expandafter\SRCOKfalse\the\output}

\begin{document}

\author{Martin Meyries}
\address{Department of Mathematics,
Karlsruhe Institute of Technology, 76128 Karlsruhe, Germany.}
\email{martin.meyries@mathematik.uni-halle.de}

\author{Mark Veraar}
\address{Delft Institute of Applied Mathematics\\
Delft University of Technology \\ P.O. Box 5031\\ 2600 GA Delft\\The
Netherlands} \email{M.C.Veraar@tudelft.nl}

\title[Traces and embeddings of anisotropic function spaces]{Traces and embeddings of anisotropic function spaces}


\thanks{The first author was supported by the project ME 3848/1-1 of the Deutsche Forschungsgemeinschaft (DFG). The second author was supported by a VENI subsidy 639.031.930 of the Netherlands Organisation for Scientific Research (NWO)}

\maketitle
\begin{abstract}
In this paper we characterize the trace spaces of a class of weighted function spaces of intersection type with mixed regularities. To a large extent we can overcome the difficulty of mixed scales by employing a microscopic improvement in Sobolev and mixed derivative embeddings with fixed integrability. We apply the general results to prove maximal $L^p$-$L^q$-regularity for the linearized, fully inhomogeneous two-phase Stefan problem with Gibbs-Thomson correction.
\end{abstract}

\section{Introduction}

In recent years the $L^p$-$L^q$-maximal regularity approach to parabolic PDEs has attracted much attention. In the influential works \cite{DHP, KuWe, We} a new theory of maximal $L^p$-regularity was founded and many classes of examples are shown to have this property.
Maximal regularity means that there is an isomorphism between the data and the solution
of the linear problem in suitable function spaces. Having established such a sharp regularity
result one can treat quasilinear problems by quite simple tools, like the contraction principle and the implicit function theorem (see \cite{Amann05,ClLi,Pruss02} and references therein).
Due to scaling invariance of PDEs one  requires $p\neq q$ for the underlying function space $L^p(L^q)$ (see e.g.  \cite[Section 3]{Can04} and \cite{Giga86}), where $p$ is the integrability in time and $q$ is for the space variable.

In the $L^p$-$L^q$-approach to linear and quasilinear parabolic problems with nonhomogeneous boundary conditions it is essential to know the precise temporal and spatial trace spaces of the unknowns. In this way different types and scales of function spaces meet and come naturally into play. For example, in the $L^p$-$L^q$-approach to the heat equation one looks for strong solutions in the parabolic Sobolev space
$$H^{1,p}(\R_+; L^q(\R^d)) \cap L^p(\R_+; H^{2,q}(\R^d)),$$
whose temporal trace space at $t = 0$ is well-known to be the Besov space $B_{q,p}^{2-2/p}(\R^d)$.
More recently, it  turned out that the spatial trace space at the coordinate $x_d = 0$ is the intersection space
\begin{equation}\label{eq:Weidemaierspace}
F_{p,q}^{1-1/(2q)}(\R_+; L^q(\R^{d-1})) \cap L^p(\R_+; B_{q,q}^{2-1/q}(\R^{d-1})),
\end{equation}
where $F_{p,q}^s$ denotes a Triebel-Lizorkin space.
The spatial trace space \eqref{eq:Weidemaierspace} was obtained in \cite{Wei02, Wei05} for $q\leq p$ and more general cases were considered in \cite{DHP07, JS08}. We conclude that the $L^p$-$L^q$-approach for already such a basic example as the heat equation with inhomogeneous boundary conditions involves three scales of function spaces.

In the case of free boundary problems or, more generally, for parabolic boundary value problems of relaxation type (see \cite{DPZ08}), a second unknown  is involved, which only lives on the boundary. For instance, for the transformed and linearized two-phase Stefan problem with Gibbs-Thomson correction, the optimal space for the boundary unknown is
\begin{equation}\label{intro-7}
 F_{p,q}^{3/2-1/(2q)}(\R_+; L^q(\R^{d-1}))\cap F_{p,q}^{1-1/(2q)}(\R_+; H^{2,q}(\R^{d-1})) \cap L^p(\R_+; B_{q,q}^{4-1/q}(\R^{d-1})),
\end{equation}
see \cite{DHP07, DK14, Kaipthesis} or Section \ref{sec:Stefan} below. The corresponding original problem is a free boundary problem which models the melting of ice, see \cite{EPS03} and the references therein. To treat the problem with nontrivial initial values one now has to determine the precise temporal trace space at $t = 0$ of this triple intersection space. If more than one boundary condition is involved, then mixed derivative (or Newton polygon) embeddings are important to  determine the optimal regularities of all boundary inhomogeneities, see \cite{DK14, DPZ08, DSS08, EPS03, Kaipthesis, MS11, MS12} and Section \ref{sec:Newton}.

Stochastic parabolic equations and Volterra integral equations (see \cite{NVWmax, Zac03, Zac05}) are further scenarios in which intersection spaces, even in an abstract form, come naturally into play. For an operator $A$ with a bounded $H^\infty$-calculus on a space $X=L^q$ with $q\geq 2$ it is shown in \cite{NVWmax} that the pathwise optimal regularity in the context of stochastic maximal $L^p$-regularity is
$$H^{s,p}(\R_+;X) \cap L^p(\R_+;D(A^{1/2-s})), \qquad s\in [0,1/2).$$
In many situations, e.g., when boundary conditions are involved, the fractional power domain $D(A^{1/2-s})$ is only a closed subspace of a function space as above. This is our motivation to study intersection spaces in an abstract form.\medskip

In a next step it is natural to introduce temporal power weights $w_{\gamma}(t) = |t|^\gamma$ for the intersection spaces. Indeed, in  many cases maximal regularity properties of parabolic problems are independent of the weight (see \cite{MS12, PS04}).  The weights yield flexibility for the initial regularity and thus a scale of phase spaces where the solution semiflow acts. This can be used to show an intrinsic smoothing effect of the parabolic problem and compactness of the semiflow, which is an important property for the investigation of the long-time behavior of solutions (see \cite{KPW, PSSS12}). As it turns out, power weights are not only important for the applications. They are in fact our main technical tool to determine temporal trace spaces, even in the unweighted case.

\medskip

In this article we can to a large extent overcome the difficulty of mixed regularity scales and study trace spaces and mixed derivative embeddings for a general class of intersection spaces.
The result  allows to characterize the regularity of the initial values in the temporally weighted $L^p$-$L^q$-approach to parabolic problems with general boundary conditions, as treated in \cite{DK14, DPZ08, MS12} (see  Remark \ref{remark-DPZ} for details).

For the trace operator $$\tr_0 u = u|_{t=0}$$ on intersection spaces we have the following result. The notation $\tr_0(\mathcal A) = \mathcal B$ means that $\tr_0:{\mathcal A} \to \mathcal B$ is continuous and surjective, and that it has a continuous right-inverse. An operator $A$ on a Banach space $X$ is called positive if $A$ is densely defined and if $\|(\lambda+A)^{-1}\|_{\calL(X)} \leq \frac{C}{1+\lambda}$ for $\lambda \geq 0$. For $\alpha > 0$ and $r\in [1,\infty]$ we let $D_A(\alpha,r) = (X,D(A^m))_{\alpha/m,r}$ be a real interpolation space, where $m > \alpha$ is an arbitrary integer (see Section \ref{Sec-interp}).

\begin{theorem}\label{thm:main3}
Let $A$ be a positive operator on a Banach space $X$, let $p\in (1,\infty)$, $q,r\in [1, \infty]$ and $\gamma \in(-1, p-1)$. Consider the weight $w_\gamma(t) = |t|^\gamma$ and  suppose that $s\in \R$ and $\alpha > 0$ satisfy $s < \frac{1+\gamma}{p} <  s+\alpha$. Then
\begin{equation}\label{intro-Ftrace}
 \emph{\text{tr}}_0 \left (F^{s+\alpha}_{p,q}(\R,w_\gamma;X) \cap F^{s}_{p,q}(\R,w_\gamma;D_A(\alpha,r))\right) = D_A\Big(s+\alpha - \frac{1+\gamma}{p},p\Big),
\end{equation}
\begin{equation}\label{intro-Btrace}
 \emph{\text{tr}}_0 \left (B^{s+\alpha}_{p,q}(\R,w_\gamma;X) \cap B^{s}_{p,q}(\R,w_\gamma;D_A(\alpha,r))\right) = D_A\Big(s+\alpha - \frac{1+\gamma}{p},q\Big).
\end{equation}
\end{theorem}

The striking point is the independence of the trace space of the so-called \emph{microscopic parameters} $q$ and $r$ in \eqref{intro-Ftrace}, and of $r$ in \eqref{intro-Btrace}. It implies that for a variety of intersection spaces the trace space is independent of the regularity scales. To be precise, assume $\mathcal A_p^{s+\alpha}$, $\mathcal B_p^{s}$ and $X_\alpha$ are such that
$$F_{p,1}^{s+\alpha} \hookrightarrow \mathcal A_p^{s+\alpha} \hookrightarrow F_{p,\infty}^{s+\alpha}, \qquad F_{p,1}^{s} \hookrightarrow \mathcal B_p^{s} \hookrightarrow F_{p,\infty}^{s}, \qquad D_A(\alpha,1) \hookrightarrow X_{\alpha}\hookrightarrow D_A(\alpha,\infty).$$
For instance, one can take a Bessel-potential space $H$, a Sobolev space $W$ or a Slobodetskii space $W^{s,p} = B_{p,p}^s = F_{p,p}^s$ for $\mathcal A$ or $\mathcal B$ (see \eqref{eq:TriebelLizorkinH}-\eqref{eq:monotony}), and a fractional power domain $D(A^\alpha)$ for $X_\alpha$. Then \eqref{intro-Ftrace} implies
$$\text{tr}_0 \left( \mathcal A_p^{s+\alpha}(\R,w_\gamma;X) \cap \mathcal B_p^{s}(\R,w_\gamma;X_{\alpha})\right)  = D_A\Big(s+\alpha - \frac{1+\gamma}{p},p\Big).$$
We emphasize that the result holds true for arbitrary Banach spaces $X$, where a Littlewood-Paley representation for Bessel-potential and Sobolev spaces is in general not available (see Remark \ref{HW}).

In the general case (see Proposition \ref{prop:extgeneral}), the continuous right-inverse for $\tr_0$ is essentially the resolvent of $A$ combined with an extension operator for the half-line $\R_+$. If $-A$ generates an analytic semigroup, then one may replace the resolvent by the semigroup, see Theorem \ref{thm:semigcase}. This also provides a new result on the regularity of semigroup orbits.

For the operator $A$ one can choose any fractional power of the shifted Laplacian $1-\Delta$ on  $$X\in \big \{H^{\beta,u}(\R^d), B_{u,v}^\beta(\R^d), F_{u,v}^\beta(\R^d)\,:\, \beta\in \R, \;u\in (1,\infty), \;v\in [1,\infty]\big \}.$$
We note that the result in particular covers the unweighted case, i.e., $\gamma = 0$. In the exponent range $\gamma\in (-1,p-1)$ the weight $w_\gamma$ belongs to the Muckenhoupt class $A_p$.\medskip

As mentioned before, the intersection spaces and their variants arise in the maximal $L^p$-$L^q$-regularity approach to deterministic and stochastic parabolic evolution equations,  see \cite{DHP07, DK14, DPZ08, DSS08, EPS03, MS12, NVWmax, PS04, Wei02, Zac03, Zac05}. In Section \ref{sec:Stefan} we apply Theorem \ref{thm:main3} to determine the temporal trace space of \eqref{intro-7} for all values of $p,q\in (1,\infty)$. Combining this with the results of \cite{DK14} for trivial initial values, we can prove maximal $L^p$-$L^q$-regularity in the parameter range $\frac{2p}{p+1} < q < 2p$ for the fully inhomogeneous linearized two-phase Stefan problem with Gibbs-Thomson correction.  The case $p = q$ was considered in \cite{DSS08} for the one phase problem, and in \cite{EPS03} for the two-phase problem.

As will be explained in Remark \ref{remark-DPZ}, Theorem \ref{thm:main3} allows to determine the temporal trace in the (weighted) $L^p$-$L^q$-approach to general parabolic initial boundary value problems, as in \cite{DPZ08}.

\medskip

Theorem \ref{thm:main3} generalizes and unifies \cite[Theorem 4.2]{MS11} in the weighted case, and  \cite[Theorem 4.5]{DSS08} and \cite[Theorem 3.6]{Zac05} (see also \cite[Theorem 3.1.4]{Zac03}) in the unweighted case. In these works only the $H$- and $B$-spaces were considered.  The continuity of $\tr_0$ was proved in  \cite[Theorem 3.6]{Zac05} under the assumption that $X$ is a UMD Banach space and that $A$ is $\mathcal R$-sectorial with $\mathcal R$-angle not larger than $\frac{\pi}{\alpha}$. The reason for these stronger assumptions is that the proof in \cite{Zac05} relies on the operator-valued Fourier multiplier result due to \cite{We}. Moreover, the proof uses a result on complex interpolation of $H$-spaces with Dirichlet boundary conditions from \cite{Se}, which is not trivial to extend to weighted and the vector-valued case.

The traces of anisotropic Besov and Triebel-Lizorkin spaces, not necessarily of intersection type, are studied in \cite{Ama09, Ber84, JS08}. However, there are only partial results on how the spaces considered there are related to the intersection spaces when $X$ and $D(A)$ are function spaces over $\R^d$ as above, see \cite[Sections 3.6-3.8]{Ama09}, \cite[Section 3.2]{DK14} and \cite[Section 5]{JS08}. In particular, the case of an operator with boundary conditions on a domain is not included there.\medskip

Theorem \ref{thm:main3} will be proved in Section \ref{sec:tracesintersections}. By direct arguments, which involve a Hardy type inequality and a difference norm for $F$-spaces, we first prove \eqref{intro-Ftrace} in a special case, say $q = 1$ and $D(A^\alpha)$ instead of $D_A(\alpha,r)$. Let us describe in more detail  where now the independence of the trace space of the microscopic parameters in \eqref{intro-Ftrace} comes from. For the outer $F$-spaces, this is essentially a consequence of the Sobolev embeddings from \cite{MeyVer1} for weighted $F$-spaces with fixed integrability parameter $p$. The embeddings state that
$$F_{p,q_0}^{s_0}(w_{\gamma_0}) \hookrightarrow F_{p,q_1}^{s_1}(w_{\gamma_1}), \qquad s_0 - \frac{1+\gamma_0}{p} = s_1 - \frac{1+\gamma_1}{p}, \qquad \gamma_0 > \gamma_1  > -1,$$
for \emph{arbitrary} parameters $q_0,q_1\in [1,\infty]$. Observe that the Sobolev regularity $s_0 - \frac{1+\gamma_0}{p}$ equals the smoothness of the trace. Since the embedding is independent of $q$, also the trace space of all intersection spaces with identical Sobolev regularity in the single spaces is independent of it. Of course, a nontrivial Sobolev embedding with fixed integrability $p$ is only possible in the presence of weights. So the flexibility in the weight exponent is the main technical tool for the independence of the outer microscopic parameter, even for the result in the unweighted case $\gamma = 0$. It is well-known that the trace space of an isotropic Triebel-Lizorkin space $F_{p,q}^s(\R^d)$ is independent of $q\in [1,\infty]$, see \cite{Tri83}. The argument with weighted Sobolev embeddings provides a new proof for this fact.

The independence of the parameter $r$ in the real interpolation space $D_A(\alpha,r)$ follows from a mixed derivative embedding with microscopic improvement in the inner scale, which we prove in Theorem \ref{thm:newton}. As explained in Remark \ref{rem:mixed-der}(i), it in particular states that
$$F_{p,q}^{s+\alpha}(\R,w_\gamma;X)\cap F_{p,q}^{s}(\R,w_\gamma;D_A(\alpha,\infty)) \hookrightarrow F_{p,q}^{s+(1-\sigma)\alpha}(\R,w_\gamma;D_A(\sigma \alpha,1)), \qquad \sigma\in (0,1).$$
Observe that the space $F_{p,q}^{s+(1-\sigma)\alpha}$ is of class $J(1-\sigma)$ between $F_{p,q}^{s}$ and $F_{p,q}^{s+\alpha}$, and that $D_A(\sigma \alpha,1)$ is of class $J(\sigma)$ between $X$ and $D_A(\alpha,\infty)$  (see \cite[Section 1.10]{Tr1} for the terminology). In this sense the embedding says that for an intersection space one may transfer smoothness from the outer scale to the inner scale. The smoothness transfer yields a microscopic improvement in the inner scale, in the sense that $D_A(\sigma \alpha,1)$ is the smallest space of class  $J(\sigma)$ between $X$ and $D_A(\alpha,\infty)$. Now, having proved the trace result for \emph{some} intermediate space, e.g., the fractional power domain $D(A^\alpha)$, the mixed derivative embedding yields the result for \emph{any} intermediate space.

For $B$- and $H$-spaces, embeddings of this type are well-known and widely used in the context of boundary value problems with inhomogeneous symbols, see \cite{DK14, DPZ08, DSS08, EPS03,  MS11, MS12}.  For the mixed derivative embeddings, a  microscopic improvement in the outer scale does not hold, as we shall show in Proposition \ref{prop:outerscale} by means of a counterexample.

The trace characterization \eqref{intro-Btrace} will be deduced in Proposition \ref{prop:B-case-inter} from \eqref{intro-Ftrace} by real interpolation. To interpolate the intersection spaces we argue as in \cite{MS11} in an operator theoretic way, relying on a classical result on operator sums due to Da Prato and Grisvard.
\medskip

In Theorem \ref{thm:main3}, for the continuity of the trace the restriction $\gamma <p-1$ can be omitted, as Theorem \ref{thm:F-embed} shows. We do not know how to remove the restriction for the surjectivity. An extension of Theorem \ref{thm:main3} to the case where $D_A(\alpha,r)$ is replaced by $D_A(\beta,r)$ with $\beta>0$ depending on the spectral angle of $A$ is given in Corollary \ref{cor:main}. We finally mention that it should be possible to generalize Theorem \ref{thm:main3} to dimensions $d\geq 2$. Given the situation of the theorem and with $w_{\gamma}(x',t) = |t|^{\gamma}$, using a Fubini argument and the methods of Section \ref{sec:tracesintersections} one can see that the trace $\tr$ at the hyperplane $\{(x',0):x'\in \R^{d-1}\}$ maps
$$F^{s+\alpha}_{p,q}(\R^d,w_{\gamma};X) \cap F^{s}_{p,q}(\R^d,w_{\gamma};D_A(\alpha,r))$$
continuously into
$$B_{p,p}^{s+\alpha- \frac{1+\gamma}{p}}(\R^{d-1};X) \cap L^p(\R^{d-1};D_A(s+\alpha - \tfrac{1+\gamma}{p},p)).$$
 We expect that this is indeed the trace space. \medskip

\emph{Notations.} In the rest of the paper $X$ will denote a general complex Banach space. No further geometric properties of $X$ are assumed. The space of bounded linear operators on $X$ is denoted by $\calL(X)$.  We let $\Schw(\R^d;X)$ be the Schwartz class of $X$-valued, smooth rapidly decreasing functions on $\R^d$, and let $\TD(\R^d;X) = \calL(\Schw(\R^d);X)$ be the space of $X$-valued tempered distributions. The Fourier transform of a distribution $f$ is denoted by $\F f$ or $\widehat{f}$. We write $X_0 \hookrightarrow X_1$ if a Banach space $X_0$ embeds continuously into another Banach space $X_1$. We let  $\N = \{1, 2, 3, \ldots\}$ and $\N_0 = \N\cup \{0\}$. Generic positive constants are denoted by $C$.

\medskip

{\em Acknowledgments.} The authors thank the anonymous referees for  helpful suggestions which lead to improvements of the results and the presentation of the paper.

\section{Preliminaries\label{sec:prel}}
\subsection{Weighted function spaces}\label{prelim-1}
We briefly recall the definitions and basic properties of vector-valued function spaces with weights from Muckenhoupt's $A_\infty$-class. For details and more references we refer to \cite[Sections 2 and 3]{MeyVer1}.

Let $p\in (1,\infty)$ and let $w:\R^d\to [0,\infty)$ be a weight, i.e., a locally integrable function which is nonzero almost everywhere. The norm of $L^p(\R^d,w;X)$ is given by
$$\|f\|_{L^p(\R^d,w;X)} = \left ( \int_{\R^d} \|f(x)\|_X^p w(x) \, dx\right)^{1/p}.$$
We also write $L^p(\R^d,w) = L^p(\R^d,w;\C)$, and $L^p(\R^d;X)$ in the unweighted case. The weight $w$ is said to belong to the Muckenhoupt class $A_p$ if
$$\sup_{Q\text{ cubes in }\R^d} \left( \frac{1}{|Q|} \int_Q w(x)\,dx\right)\left ( \frac{1}{|Q|} \int_Q w(x)^{-1/(p-1)}\,dx\right)^{p-1} < \infty.$$
One further sets $A_\infty = \bigcup_{p> 1} A_p$. For the general properties of the $A_p$-classes we refer to \cite[Chapter 9]{GraModern} and \cite[Chapter V]{Stein93}.

As a special case, in the present work we mainly consider power weights $$w_\gamma(t) = |t|^\gamma, \qquad t\in \R,$$ on the real line. Here we have $w_\gamma \in A_p$ if and only if $\gamma \in (-1,p-1)$, and thus $w_\gamma \in A_\infty$ if and only if $\gamma>-1$ (see \cite[Example 9.1.7]{GraModern}).

Let $\Phi(\R^d)$ be the set of all sequences $(\varphi_k)_{k\geq 0} \subset \Schw(\R^d)$ such that
\begin{align}\label{eq:defPhisequence}
\wh{\varphi}_0 = \wh{\varphi}, \qquad \wh{\varphi}_1(\xi) = \wh{\varphi}(\xi/2) - \wh{\varphi}(\xi), \qquad \wh{\varphi}_k(\xi) = \wh{\varphi}_1(2^{-k+1} \xi), \quad k\geq 2, \qquad \xi\in \R^d,
\end{align}
where the Fourier transform $\wh{\varphi}$ of the generating function $\varphi\in \Schw(\R^d)$ satisfies
\begin{equation}\label{gen-func}
 0\leq \wh{\varphi}(\xi)\leq 1, \quad  \xi\in \R^d, \qquad  \wh{\varphi}(\xi) = 1 \ \text{ if } \ |\xi|\leq 1, \qquad  \wh{\varphi}(\xi)=0 \ \text{ if } \ |\xi|\geq \frac32.
\end{equation}
For  $(\varphi_k)_{k\geq 0} \in \Phi(\R^d)$ and $f\in \TD(\R^d;X)$ we let
$$S_k f = \varphi_k * f = \F^{-1} ( \wh{\varphi}_k \wh{f}).$$
Given $p \in (1,\infty)$, $q\in [1,\infty]$ and $w\in A_\infty$,  for $f\in {\mathscr S}'(\R^d;X)$ we set
\[ \|f\|_{B_{p,q}^s (\R^d,w;X)} = \Big\| \big( 2^{sk}S_k f\big)_{k\geq 0} \Big\|_{\ell^q(L^p(\R^d,w;X))},\]
\[ \|f\|_{F_{p,q}^s (\R^d,w;X)} = \Big\| \big( 2^{sk}S_k f\big)_{k\ge 0} \Big\|_{L^p(\R^d,w;\ell^q(X))}, \]
\[\|f\|_{H^{s,p}(\R^d,w;X)} = \|\F^{-1} [(1+|\cdot|^2)^{s/2} \wh{f} ]\|_{L^p(\R^d,w;X)}.\]
These norms define the Besov space $B_{p,q}^s (\R^d,w;X)$, the Triebel-Lizorkin space $F_{p,q}^s (\R^d,w;X)$, and the Bessel-potential space $H^{s,p}(\R^d,w;X)$, respectively, which are all Banach spaces. Any other $(\psi_k)_{k\geq 0} \in \Phi(\R^d)$ leads to an equivalent norm on the $B$- and $F$-spaces. Observe that $B_{p,p}^s = F_{p,p}^s$ by Fubini's theorem. If $s >0$ with $s\notin \N$, then one sets
$$W^{s,p}(\R^d,w;X) = B_{p,p}^s (\R^d,w;X),$$
and for $m\in \N_0$,
\[\|f\|_{W^{m,p}(\R^d,w;X)} = \sum_{|\alpha|\leq m} \|D^{\alpha} f\|_{L^p(\R^d,w;X)},\]
where the derivatives $D^{\alpha}$ are taken in a distributional sense. These norms define the Slobodetskii and the Sobolev spaces, respectively.

\begin{remark}\label{HW}
Note that $L^p(\R^d,w;X) = H^{0,p}(\R^d,w;X) = W^{0,p}(\R^d,w;X)$. However, one has $H^{1,p}(\R^d;X) = W^{1,p}(\R^d;X)$ if and only if $X$ has the UMD property (see \cite{Ama95,McC84,Zim89}). Moreover, $L^p(\R^d;X) = F_{p,2}^0(\R^d;X)$ if and only if $X$ can be renormed as a Hilbert space (see \cite{HaMe96} and \cite[Remark 7]{SchmSiunpublished}).
\end{remark}

\subsection{Embeddings} \label{sec:embeddings} Each of the above spaces embeds continuously into $\TD(\R^d;X)$. This can be seen as in the proof of \cite[Theorem 2.3.3]{Tri83} (using H\"older's inequality instead of Nikolskii's inequality and \cite[Lemma 4.5]{MeyVer1} to get rid of the weight). Conversely, $\Schw(\R^d;X)$ embeds continuously into each of the above spaces, where this is a dense embedding if $q\in [1,\infty)$ (see \cite[Lemma 3.8]{MeyVer1}).

There are elementary embeddings between the function spaces, see \cite[Propositions 3.11 and 3.12]{MeyVer1}. For $s\in \R$ and $m\in \N_0$ we shall make particular use of
\begin{align}\label{eq:TriebelLizorkinH}
F^{s}_{p,1}(\R^d,w;X)  \hookrightarrow H^{s,p}(\R^d,w;X) \hookrightarrow F^{s}_{p,\infty}(\R^d,w;X),\\
\label{eq:TriebelLizorkinW}
F^{m}_{p,1}(\R^d,w;X)  \hookrightarrow W^{m,p}(\R^d,w;X) \hookrightarrow F^{m}_{p,\infty}(\R^d,w;X).
\end{align}
Here, the embeddings of $F^{s}_{p,1}(\R^d,w;X)$ and $F^{m}_{p,1}(\R^d,w;X)$ are valid for all $w\in A_\infty$. The embeddings into $F^{s}_{p,\infty}$ and $F^{m}_{p,\infty}$ are valid for $w\in A_p$, and in fact a local $A_p$-condition is necessary for them to hold (see \cite[Remark 3.13]{MeyVer1}). For power weights $w_\gamma(t) = |t|^\gamma$ on $\R$ this condition is equivalent to the usual $A_p$-condition.

For all $q\in [1, \infty]$ one further has
\begin{equation}\label{eq:scaleBF}
B_{p,\min\{p,q\}}^{s} (\R^d,w;X)\hookrightarrow F_{p,q}^{s} (\R^d,w;X) \hookrightarrow B_{p,\max\{p,q\}}^{s} (\R^d,w;X),
\end{equation}
and if $1\leq q_0\leq q_1\leq \infty$, then
\begin{equation}\label{eq:monotony}
 B_{p,q_0}^s (\R^d,w;X)\hookrightarrow B_{p,q_1}^s (\R^d,w;X), \qquad F_{p,q_0}^s (\R^d,w;X)\hookrightarrow F_{p,q_1}^s (\R^d,w;X),
\end{equation}
by monotonicity of the $\ell^q$-spaces. Hence $B_{p,1}^s$ is the smallest and $B_{p,\infty}^s$ is the largest of the spaces $F$, $B$, $H$ and $W$ for fixed $s$ and $p$.

Crucial for our investigations of traces are the following Sobolev type embeddings for Triebel-Lizorkin spaces one the real line with power weights $w_\gamma(t) = |t|^\gamma$, which are a special case of \cite[Theorem 1.2]{MeyVer1}.

\begin{theorem}\label{thm:SobB} Let $1<p_0\leq p_1 < \infty$, $q_0,q_1\in [1,\infty]$, $s_0 > s_1$ and $\gamma_0,\gamma_1>-1$. Suppose that
\begin{equation*}
\frac{\gamma_0}{p_0} \geq \frac{\gamma_1}{p_1}, \qquad \text{and} \qquad s_0 -\frac{1+\gamma_0}{p_0} =  s_1 - \frac{1+\gamma_1}{p_1}.
\end{equation*}
Then one has the continuous embedding
\begin{equation*}
F^{s_0}_{p_0, q_0}(\R,w_{\gamma_0}; X)\hookrightarrow F^{s_1}_{p_1, q_1}(\R,w_{\gamma_1}; X).
\end{equation*}
\end{theorem}
It is rather surprising that one can take $p_0 = p_1$ in the above result and still have the independence of the microscopic parameters $q_0$ and $q_1$.

\subsection{Characterization of weighted $F$-spaces by differences\label{sec:charact}} For an integer $m\geq 1$ let
\[\Delta_h^m f(x) = \sum_{l =0}^m {{m}\choose{l }} (-1)^l f(x+(m-l)h), \qquad x,h\in \R^d.\]
For a weight $w$ and $f\in L^p(\R^d,w;X)$ define
\[[f]_{F^s_{p,q}(\R^d,w;X)}^{(m)} = \Big\|\Big(\int_0^\infty t^{-sq}  \Big(t^{-d}\int_{|h|\leq t} \|\Delta_h^m f(\cdot)\|_X \, dh \Big)^{q}\, \frac{dt}{t}\Big)^{1/q}\Big\|_{L^p(\R^d,w)},\]
with the usual modification if $q=\infty$, and
\[\nn f\nn_{F^s_{p,q}(\R^d,w;X)}^{(m)} = \|f\|_{L^p(\R^d,w;X)} + [f]_{F^s_{p,q}(\R^d,w;X)}^{(m)}.\]
We also write $\nn f\nn_{F^s_{p,q}(\R^d,w;X)}$ for $\nn f\nn_{F^s_{p,q}(\R^d,w;X)}^{(m)}$ if there is no danger of confusion. Note that if $q=1$, then Fubini's theorem yields
\begin{equation}\label{eq:Fnormq=1}
 [f]_{F^s_{p,1}(\R^d,w;X)}^{(m)} = c_d \Big\|\int_{\R^d} |h|^{-s-d} \|\Delta_h^m f(\cdot)\|_X \, dh \Big\|_{L^p(\R^d,w)}.
\end{equation}

One can extend a well-known result on the equivalence of norms for $F$-spaces to the weighted case (cf. \cite[Proposition 6]{SchmSiunpublished}, \cite[Section 2.5.10]{Tri83} and \cite[Theorem 6.9]{Triebel3}). A similar characterization is valid for $B$-spaces.
\begin{proposition}\label{prop:Lpsmoothness}
Let $s>0$, $p\in (1, \infty)$, $q\in [1, \infty]$ and $w\in A_p$. Let $m\geq 1$ be an integer such that $m>s$.
Then there is a constant $C>0$ such that for all $f\in L^p(\R^d,w;X)$
\begin{equation}\label{eq:equivnorm}
C^{-1} \|f\|_{F^s_{p,q}(\R^d,w;X)}\leq \nn f\nn_{F^s_{p,q}(\R^d,w;X)}^{(m)} \leq C\|f\|_{F^s_{p,q}(\R^d,w;X)},
\end{equation}
whenever one of these expressions is finite.
\end{proposition}

\subsection{Positive operators and interpolation} \label{Sec-interp}
We recall some standard definitions and results on positive operators. For detailed expositions we refer to \cite{Ama95, DHP, Haase:2, Lun09, Tr1}.

Let $A$ with domain $D(A)$ be a closed and densely defined operator on $X$. Then $A$ is called a \emph{positive operator} if $[0,\infty)$ is contained in the resolvent set of $-A$ and
$$ \| (\lambda + A)^{-1}\|_{\calL(X)}\leq  \frac{C}{1+\lambda}, \qquad \lambda\geq 0.$$
For $\alpha \in \R$ the \emph{fractional power} $A^\alpha$ of a positive operator $A$ can be defined as in \cite[Section 1.15]{Tr1}.

For $\theta \in (0,1)$ and $p\in [1,\infty]$ the real and the complex interpolation functor are denoted by $(\cdot, \cdot)_{\theta,p}$ and $[\cdot, \cdot]_\theta$, respectively. For a positive operator $A$ and $\alpha > 0$ one sets
$$D_A(\alpha,p) = (X, D(A^m))_{\alpha/m,p},$$
where $m > \alpha$ is an arbitrary integer. It follows from reiteration (see \cite[Theorem 1.15.2]{Tr1}) that $D_A(\alpha,p)$ is independent of the choice of $m$, and further that
\begin{equation}\label{reiteration}
 (X, D_A(\alpha,p))_{\theta,q} = D_A(\theta\alpha,q), \qquad q\in [1,\infty],
\end{equation}
see \cite[Section 1.3]{Lun09}.   By \cite[Theorem 1.14.3]{Tr1}, an equivalent norm for $D_A(\alpha,p)$ is given by
\begin{equation}\label{eq:realinterResolvent}
y \mapsto \Big(\int_0^\infty \sigma^{\alpha p} \|( A(\sigma+A)^{-1})^{m} y\|^p \, \frac{d\sigma}{\sigma}\Big)^{1/p},
\end{equation}
with the usual modification for $p=\infty$. Here again $m > \alpha$ is an arbitrary integer. In \cite[Theorem 1.15.2]{Tr1} it is shown that for all $\alpha, \beta> 0$ the operator $A^\beta$ is an isomorphism
\begin{equation}\label{Aiso}
 A^\beta: D_A(\alpha+\beta,p) \to D_A(\alpha,p), \qquad A^\beta: D(A^{\alpha+\beta}) \to D(A^\alpha).
\end{equation}
The space $D_A(\alpha,1)$ is the smallest and $D_A(\alpha,\infty)$ is the largest intermediate space of order $\alpha$ for $A$, in the sense that if $X_\alpha$ is of class $J(\alpha/m)\cap K(\alpha/m)$ between $X$ and $D(A^m)$, then
$$D_A(\alpha,1) \hookrightarrow X_\alpha \hookrightarrow D_A(\alpha,\infty),$$
see  \cite[Section 1.10]{Tr1}.

\section{Mixed derivative embeddings with microscopic improvement\label{sec:Newton}}

The main result of this section is the following embedding of mixed derivative type.
\begin{theorem} \label{thm:newton}
Let $X_0$, $X_1$ be an interpolation couple. Let $\theta\in (0,1)$ and $X_{\theta}$ be a Banach space such that $X_0\cap X_1 \subseteq X_{\theta}\subseteq X_0+X_1$ and
\begin{align}\label{eq:ordertheta}
\|x\|_{X_{\theta}} \leq C \|x\|_{X_0}^{1-\theta} \|x\|^{\theta}_{X_1}, \qquad x\in X_0 \cap X_1.
\end{align}
Assume $p_0, p_1, p\in (1,\infty)$, $q_0, q_1,q\in [1,\infty]$, $w_0,w_1\in A_{\infty}$ and a weight $w$  satisfy
\begin{align*}
\frac{1-\theta}{p_0} + \frac{\theta}{p_1} = \frac1p, \qquad \frac{1-\theta}{q_0} + \frac{\theta}{q_1} = \frac1q,
\qquad w_0^{1-\theta} w_1^{\theta} = w.
\end{align*}
Let further $s\in \R$, $\alpha > 0$ and $ \mathcal A \in \{F,B\}$. Then
\[\mathcal A_{p_0,q_0}^{s+\alpha}(\R^d,w_0;X_0) \cap \mathcal A_{p_1,q_1}^{s}(\R^d,w_1;X_1) \hookrightarrow \mathcal A_{p,q}^{s+(1-\theta)\alpha}(\R^d,w;X_{\theta}),\]
and for all $f\in \mathcal A_{p_0,q_0}^{s+\alpha}(\R^d,w_0;X_0) \cap \mathcal A_{p_1,q_1}^{s}(\R^d,w_1;X_1)$ one has
\begin{align*}
\|f\|_{{\mathcal A}^{s+(1-\theta)\alpha}_{p,q}(\R^d,w;X_{\theta})}  \leq C \|f\|_{{\mathcal A}^{s+\alpha}_{p_0,q_0}(\R^d,w_0;X_{0})}^{1-\theta}  \|f\|_{{\mathcal A}^{s}_{p_1,q_1}(\R^d,w_1;X_{1})}^{\theta}.
\end{align*}
\end{theorem}
Recall from \cite[Section 1.10]{Tr1} that if a space $X_{\theta}$ satisfies \eqref{eq:ordertheta}, then it belongs to the class $J(\theta)$ between $X_0$ and $X_1$, which in turn is equivalent to $(X_0,X_1)_{\theta,1} \hookrightarrow X_{\theta}$. It is in particular satisfied for the real interpolation spaces $(X_{0}, X_1)_{\theta,u}$ with $u\in [1,\infty]$ and the complex interpolation space $[X_0,X_1]_{\theta}$ (cf. \cite[Theorem 1.10.3/1]{Tr1}). In this sense the embedding allows to transfer smoothness between the inner and the outer scale.

There are also mixed derivative type embeddings available for $\mathcal A \in \{H,W\}$ in the outer scale, see \cite{DK14, DSS08, MS11}. These are based on an abstract result due to \cite{Sob75} concerning the boundedness of $A^{1-\theta}B^{\theta}(A+B)^{-1}$ for resolvent commuting positive operators $A$ and $B$, which typically have to satisfy assumptions of Dore-Venni type (see also \cite[Lemma 4.1]{DSS08}, \cite{EPS03} or \cite[Proposition 1.1]{MS11}).

In the case of $B$- and $F$-spaces we will deduce the result directly from the definitions.

\begin{proof}[Proof of Theorem \ref{thm:newton}] Let us consider the case $\mathcal A = F$. First note that $w\in A_{\infty}$ (see \cite[Exercise 9.1.5]{GraModern}). By assumption we have
\begin{align*}
2^{(s+(1-\theta)\alpha)n} \|S_n f\|_{X_{\theta}}& \leq C \big(2^{(s+\alpha)n} \|S_n f\|_{X_{0}}\big)^{1-\theta} \big(2^{sn} \|S_n f\|_{X_{1}}\big)^{\theta}.
\end{align*}
Taking $L^p(\R^d,w;\ell^q)$-norms and using H\"older's inequality twice, we find that
\begin{align*}
\|f\|_{F^{s+(1-\theta)\alpha}_{p,q}(\R^d,w;X_{\theta})} & = \big\|(2^{(s+(1-\theta)\alpha)n} \|S_n f\|_{X_{\theta}})_{n\geq 0}\big\|_{L^{p}(\R^d,w;\ell^{q})}
\\ & \leq C \big\|\big(2^{(s+\alpha)n} \|S_n f\|_{X_{0}}\big)_{\geq 0}\big\|_{L^{p_0}(\R^d,w_0;\ell^{q_0})}^{1-\theta} \big\|\big(2^{sn} \|S_n f\|_{X_{1}}\big)_{n\geq 0}\big\|_{L^{p_1}(\R^d,w_1;\ell^{q_1})}^{\theta}
\\ & = C\|f\|_{F^{s+\alpha}_{p_0,q_0}(\R^d,w_0;X_{0})}^{1-\theta}  \|f\|_{F^{s}_{p_1,q_1}(\R^d,w_1;X_{1})}^{\theta}.
\end{align*}
The asserted embedding now follows from Young's inequality. The case $\mathcal A = B$ is proved in the same way.
\end{proof}

We comment on special cases of interest.
\begin{remark} \label{rem:mixed-der}\
\begin{enumerate}[(i)]
%
\item If $A$ is a positive operator (see Section \ref{Sec-interp}) on $X_0 = X$ with $X_1 = D(A)$, then one can take $X_\theta = D(A^\theta)$ or $X_\theta = D_A(\theta, 1) = (X, D(A))_{\theta,1}$ (see \cite[Theorem 1.15.2]{Tr1}). The latter is the smallest space of class $J(\theta)$ between $X$ and $D(A)$.  In this sense we obtain a microscopic improvement in the inner scale. More generally, employing $D_A(\theta \beta,1) = (X, D_A(\beta,\infty))_{\theta,1}$
for $\beta > 0$ (see \eqref{reiteration}), we get
\[\mathcal A_{p,q}^{s+\alpha}(\R^d;X) \cap \mathcal A_{p,q}^{s}(\R^d;D_A(\beta,\infty)) \hookrightarrow \mathcal A_{p,q}^{s+(1-\theta) \alpha}(\R^d;D_A(\theta\beta,1)).\]

\item The microscopic improvement becomes even more transparent when considering function spaces in the inner regularity scale. Recall from Section \ref{sec:embeddings} that $B_{r,1}^t$ is the smallest and $B_{r,\infty}^t$ is the largest of the $B$-, $F$-, $H$- and $W$-spaces for fixed $t$ and $r$. Since for $\beta >0$ we have (see \cite[Theorem 2.4.1]{Tr1})
$$(B_{r,\infty}^{t}, B_{r,\infty}^{t+\beta})_{\theta,1} = B_{r,1}^{t+\theta\beta},$$
the theorem yields
$$\mathcal A_{p,q}^{s+\alpha}(\R^d; B_{r,\infty}^{t}) \cap \mathcal A_{p,q}^{s}(\R^d;B_{r,\infty}^{t+\beta}) \hookrightarrow \mathcal A_{p,q}^{s+(1-\theta)\alpha}(\R^d; B_{r,1}^{t+\theta\beta}).$$
For instance, this implies
$$B_{p,q}^{s+\alpha}(\R; L^r) \cap B_{p,q}^{s}(\R; W^{\beta,r}) \hookrightarrow B_{p,q}^{s+(1-\theta)\alpha}(\R; H^{\theta\beta,r}).$$
In this direction, another interesting case is
$$X_0 = F_{r_0,\infty}^{t}, \qquad  X_1 = F_{r_1,\infty}^{t+\beta}, \qquad X_{\theta} = F_{r,1}^{t+\theta\beta},$$
where $\frac{1}{r}= \frac{1-\theta}{r_0} + \frac{\theta}{r_1}$. Here a Gagliardo-Nirenberg type inequality (see, e.g., \cite[Proposition 5.1]{MeyVer1}) ensures that \eqref{eq:ordertheta} holds true.
\end{enumerate}
\end{remark}

At this point it is natural to ask for a microscopic improvement in the outer regularity scale. The next result implies that this does not hold, in general.
\begin{proposition}\label{prop:outerscale}
Let $s,t\in \R$, $\alpha, \beta>0$, $p,r\in (1, \infty)$, $1\leq u<q\leq \infty$ and $\theta\in [0,1]$. Then
\[F^{s+\alpha}_{p,q}(\R;B^{t}_{r,1}(\R))\cap F^{s}_{p,q}(\R;B^{t+\beta}_{r,1}(\R))\nsubseteq F^{s+(1-\theta)\alpha}_{p,u}(\R;B^{t+\theta \beta}_{r,\infty}(\R)).\]
\end{proposition}

As a consequence of the monotonic properties \eqref{eq:scaleBF} of the function spaces and the elementary embeddings \eqref{eq:TriebelLizorkinH} and \eqref{eq:TriebelLizorkinW}, an inclusion as above for intersection spaces where the Besov spaces in the inner scale are replaced by any $F$-, $B$-, $H$- or $W$-space with the same smoothness and integrability parameter does not hold as well. Furthermore, let $u\in [1, q)$. Then the proposition implies that, for a positive operator $A$, in general one has
\[F^{s+\alpha}_{p,q}(\R;X)\cap F^{s}_{p,q}(\R;D(A))\nsubseteq F^{s+(1-\theta)\alpha}_{p,u}(\R;D(A^{\theta})).\]
Indeed, for example one can take
$A = 1-\Delta$ on $L^2(\R)$ with $D(A) = H^{2,2}(\R)$ and $D(A^{\theta}) = H^{2\theta,2}(\R)$.

\begin{proof}[Proof of Proposition \ref{prop:outerscale}]
In order to obtain a contradiction, assume that the inclusion holds, which is then a continuous embedding by a closed graph argument.

Let $R = 2^{\frac{\alpha}{\beta}}$.
Recall from \cite[Remark 2.3.1/3]{Tri83} that for $t\in \R$ and $z\in [1, \infty]$, the norm of $B^{t}_{r,z}(\R)$ is equivalent to
\[f\mapsto \|(R^{tn} \psi_n*f)_{n\geq 0}\|_{\ell^z(L^{r})},\]
where $(\widehat{\psi}_n)_{n\geq 0}$ is a decomposition of unity with the $2^n$-factor replaced by $R^n$ (see the definition of $\Phi(\R)$ in Section \ref{prelim-1}). Let further $(\varphi_n)_{n\geq 0}\in \Phi(\R)$.

We may assume that there is a small $\delta >0$ such that $\widehat{\varphi}_n = 1$ and $\widehat{\varphi}_{j} = 0$ for $j\neq n$  on $[2^{n}-\delta, 2^{n}+\delta]$, and $\widehat{\psi}_n = 1$ and $\widehat{\psi}_{j} = 0$ for $j\neq n$ on $[R^{n}-\delta, R^{n}+\delta]$, for all $n\geq 1$.
Fix a sequence of real numbers $(a_{n})_{n\geq 1}$ of which only finitely many are nonzero. Let $f:\R\to B^{t+\beta}_{r,1}(\R)$ be defined by
\[f(x)(y) = \sum_{n\geq 1} a_n (\F^{-1} \one_{[2^n-\delta,2^n+\delta]})(x) (\F^{-1} \one_{[R^n-\delta,R^n+\delta]})(y) = \sum_{n\geq 1} a_n e^{2\pi i2^n x} e^{2\pi i R^n y}  \zeta(x) \zeta(y),\]
where $\zeta(x) = \frac{1}{2\pi i x}( e^{2\pi i x \delta} - e^{-2\pi i x \delta})$ is independent of $n$. Note that $\zeta \in L^\sigma(\R)$ for each $\sigma > 1$. Let $Y= B^t_{r,z}(\R)$ and $s\in \R$. Then one has
\[\|f\|_{F^{s}_{p,q}(\R;Y)} = \|(2^{sj} \varphi_j*f)_{j\geq 0}\|_{L^p(\R;\ell^q(Y))}.\]
Here $\varphi_0*f =0$, and for each $j\geq 1$ and $x\in \R$,
\[\|\varphi_j*f(x)\|_Y = \|(R^{tn} \psi_n * [\varphi_j*f(x)])_{n\geq 0}\|_{\ell^z(L^{r}(\R))} = R^{tj} |a_j|  |\zeta(x)| \|\zeta\|_{L^{r}(\R)},\]
where in the second expression the first convolution is with respect to $y$ and the second convolution is with respect to $x$.
It follows that
\begin{equation*}
\|f\|_{F^s_{p,q}(\R;B^t_{r,z}(\R) )} =  \|f\|_{F^s_{p,q}(\R;Y)} = C_{p,r} \|(2^{sj} R^{tj}a_j)_{j\geq 1}\|_{\ell^q},
\end{equation*}
where $C_{p,r} =  \|\zeta\|_{L^{p}(\R)} \|\zeta\|_{L^{r}(\R)}$.
Therefore, using $R =2^{\frac{\alpha}{\beta}}$,
\begin{align*}
\|f\|_{F^{s+(1-\theta)\alpha}_{p,u}(\R;B^{t+\theta\beta}_{r,\infty}(\R))}  & = C_{p,r} \|(2^{(s+(1-\theta)\alpha)j} R^{(t+\theta\beta)j} a_j)_{j\geq 1}\|_{\ell^u} = C_{p,r} \|(2^{(s+ \frac{\alpha}{\beta} t + \alpha)j}  a_j)_{j\geq 1}\|_{\ell^u},
\\
\|f\|_{F^{s+\alpha}_{p,q}(\R;B^{t}_{r,1}(\R))} &= C_{p,r}  \|(2^{(s+\alpha)j }R^{tj} a_j)_{j\geq 1}\|_{\ell^q} = C_{p,r} \|2^{(s+ \alpha + \frac{\alpha}{\beta} t )j} a_j)_{j\geq 1}\|_{\ell^q},
\\  \|f\|_{F^{s}_{p,q}(\R;B^{t+\beta}_{r,1}(\R))} & = C_{p,r}  \|(2^{sj} R^{(t+\beta)j} a_j)_{j\geq 1}\|_{\ell^q} = C_{p,r} \|(2^{(s+ \frac{\alpha}{\beta} t + \alpha)j}  a_j)_{j\geq 1}\|_{\ell^q}.
\end{align*}
All the sequences in the above norms coincide. Hence the continuity of the embedding yields $\ell^q\hookrightarrow \ell^u$, which is false.
\end{proof}

\section{Proof of Theorem \ref{thm:main3}: traces of weighted anisotropic spaces\label{sec:tracesintersections}}

Let $A$ be a  positive operator on $X$ and $w_\gamma(t)=|t|^\gamma$ with $\gamma\in (-1,p-1)$. In this section we will first show the characterization \eqref{intro-Ftrace} in Theorem \ref{thm:main3}. This is to prove that the image of the trace operator $\tr_0 u = u|_{t=0}$ for a space
$$F_{p,q}^{s+\alpha}(\R,w_\gamma;X) \cap F_{p,q}^{s}(\R,w_\gamma; D_A(\alpha,r)), \qquad  s < \frac{1+\gamma}{p} < s+\alpha, \quad q,r\in [1,\infty],$$
is the real interpolation space $D_A(\theta,p)$, where $\theta = s+\alpha-\frac{1+\gamma}{p}$.

The operator $\tr_0$ is defined on an intersection space as above in the following sense. For $q\in [1,\infty)$ it can be seen as in \cite[Lemma 3.8]{MeyVer1} that $\Schw(\R; D_A(\alpha,r))$ is a dense subset. The trace of such functions is defined in the classical sense and may be extended to an intersection space by a corresponding norm estimate. For  $q=\infty$ we find a suitable larger intersection space where the trace can again be defined by density and whose trace space turns out to be the correct one for the original space.

We will obtain  \eqref{intro-Ftrace} as a consequence of Theorem \ref{thm:F-embed} and Proposition \ref{prop:extgeneral}. In fact, in Theorem \ref{thm:F-embed} we prove that the trace $\tr_0$ may continuously be extended to
\[\tr_0:F^{s+\alpha}_{p,\infty}(\R,w_\gamma;X) \cap F^{s}_{p,\infty}(\R,w_\gamma;D_A(\alpha,\infty)) \to D_A(\theta,p).\]
Since $F^{s}_{p,q}\hookrightarrow F^{s}_{p,\infty}$ and $D_{A}(\alpha, r) \hookrightarrow D_A(\alpha,\infty)$, this shows the continuity for all $q$ and $r$. This result actually holds for all $\gamma > -1$.

To prove the surjectivity of $\tr_0$, in Proposition \ref{prop:extgeneral} we show that there is a continuous right-inverse of $\tr_0$ mapping
\[D_A(\theta,p)\to F^{s+\alpha}_{p,1}(\R,w_\gamma;X) \cap F^{s}_{p,1}(\R,w_\gamma;D_A(\alpha,1)).\]
As above, since $F^{s}_{p,1}\hookrightarrow F^{s}_{p,q}$ and $D_A(\alpha,1) \hookrightarrow D_A(\alpha,r)$, this gives a continuous right-inverse for all $q$ and $r$. The right-inverse is essentially the resolvent of $A$ combined with an extension operator for the half-line $\R_+$.

We emphasize that also in the unweighted case $\gamma = 0$ the proofs below make heavy use of power weights.

\begin{theorem}\label{thm:F-embed}
Let $A$ be a positive operator on $X$, let $p\in (1,\infty)$ and $\gamma > -1$. Suppose that $s\in \R$ and $\alpha > 0$ satisfy $s < \frac{1+\gamma}{p} <  s+\alpha$, and set $\theta = s+\alpha - \frac{1+\gamma}{p}$. Then the trace operator $\emph{\text{tr}}_0 u = u|_{t=0}$ extends to a continuous map
\begin{equation}\label{trace-cont}
 F^{s+\alpha}_{p,\infty}(\R,w_\gamma;X) \cap F^{s}_{p,\infty}(\R,w_\gamma;D_A(\alpha,\infty)) \to D_A(\theta,p).
\end{equation}
\end{theorem}

\begin{proof} In the Steps 1--4 we prove \eqref{trace-cont} with $F^{t}_{p,\infty}$ replaced by $F^{t}_{p,1}$ for $t\in \{s, s+\alpha\}$. In Step 5 we show the general result \eqref{trace-cont}.

We will use a classical Hardy-Young inequality (see \cite[p. 245-246]{HLP}), stating that for all measurable functions $f:\R_+\to \R_+$ and all $\beta >0$ we have
\begin{equation}\label{eq:HardyYoung}
\int_0^\infty \sigma^{-\beta p-1} \Big(\int_0^\sigma f(\tau) \, \, d\tau\Big)^p \, d\sigma \leq \beta^{-p} \int_0^\infty \sigma^{-\beta p -1+p} f(\sigma)^p\, d\sigma.
\end{equation}

{\em Step 1.} First assume
\begin{equation}\label{eq:assumptionStep1}
0 < s < \frac{1+\gamma}{p} < s+\alpha<1,
\end{equation}
such that $\theta\in (0,1)$.
The idea of the following argument is due to \cite[Lemma 11]{DiB84}. Let $u\in \Schw(\R; D_A(\alpha,\infty))$. Writing $u(t) - u(\tau) = \int_\tau^t u'(\xi) \, d\xi$, it is straight forward to check that for all $\sigma > 0$ we have
\[u(0) = \sigma^{-1}\int_0^\sigma u(\tau)\,d \tau - \int_0^\sigma t^{-2} \int_0^t u(t)-u(\tau) \, d\tau  \, dt.\]
This representation and the equivalent norm \eqref{eq:realinterResolvent} with $m=1$ yield
\begin{align*}
\|u(0)\|_{D_A(\theta,p)}  & \leq C (T_1+ T_2),
\end{align*}
where
\begin{align*}
T_1^p &= \int_0^\infty \sigma^{-\theta p-1}\Big\|\sigma^{-1}\int_0^\sigma A (\sigma^{-1} + A)^{-1} u(\tau)\,d \tau \Big\|^p \, d\sigma,\\
T_2^p &=\int_0^\infty \sigma^{-\theta p-1} \Big\|\int_0^\sigma t^{-2} \int_0^t A(\sigma^{-1} + A)^{-1}(u(t)-u(\tau)) \, d\tau  \, dt\Big\|^p \, d\sigma.
\end{align*}
We estimate $T_1$. By \cite[Theorem 1.14.2]{Tr1} we have
\[\|A(\sigma^{-1}+A)^{-1}x\|\leq C \sigma^{\alpha}\|x\|_{D_A(\alpha,\infty)}, \qquad \sigma > 0.\]
Together with \eqref{eq:HardyYoung} for $\beta = s-\frac{1+\gamma}{p} +1>0$ this implies
\begin{align*}
T_1^p  & \leq \int_0^\infty \sigma^{-\theta p-1-p} \Big(\int_0^\sigma \|A(\sigma^{-1}+A)^{-1}u(\tau)\| \,d \tau \Big)^{p} \, d\sigma
\\
& \leq C \int_0^\infty \sigma^{-(s-\frac{1+\gamma}{p} +1)p-1}  \Big(\int_0^\sigma \|u(\tau)\|_{D_A(\alpha,\infty)} \,d \tau \Big)^{p} \, d\sigma
\\ & \leq C\int_0^\infty \sigma^{ \gamma-sp}  \|u(\sigma)\|_{D_A(\alpha,\infty)}^p \, d\sigma \leq C \|u\|_{L^p(\R,w_{\gamma-sp}; D_A(\alpha,\infty))}^p.
\end{align*}
Now the elementary embedding \eqref{eq:TriebelLizorkinH} and the Sobolev embedding from Theorem \ref{thm:SobB} give the following inequality of Hardy type,
\begin{align*}
\|u\|_{L^p(\R,w_{\gamma-sp}; D_A(\alpha,\infty))}\leq C \|u\|_{F^0_{p,1}(\R,w_{\gamma-sp}; D_A(\alpha,\infty))} \leq C \|u\|_{F^{s}_{p,1}(\R,w_\gamma;D_A(\alpha, \infty))},
\end{align*}
where we used that $\gamma-sp > -1$. Therefore,
$$
T_1 \leq C\|u\|_{F^{s}_{p,1}(\R,w_{\gamma};D_A(\alpha, \infty))}.
$$
For the estimate of $T_2$ we observe that
\[\|A(\sigma^{-1}+A)^{-1}x\| \leq C \|x\|, \qquad \sigma > 0.\]
Set $f(t) = t^{-2} \int_0^t \|u(t)-u(\tau)\| \, d\tau$. Then from \eqref{eq:HardyYoung}, Proposition \ref{prop:Lpsmoothness} and \eqref{eq:Fnormq=1} we obtain
\begin{align*}
T_2^p & \leq C \int_0^\infty \sigma^{-\theta p-1} \Big(\int_0^\sigma  f(t)  \, dt\Big)^p \, d\sigma \leq C\int_0^\infty \sigma^{-\theta p-1+p} f(\sigma)^p  \, d\sigma \\
& = C\int_0^\infty \sigma^{-\theta p-1-p} \Big(\int_{-\sigma}^0 \|u(\sigma)-u(\sigma+h)\| \, dh\Big)^{p}   \, d\sigma\\
 & \leq C\int_0^\infty \sigma^{\g} \Big(\int_{-\sigma}^0 |h|^{-s-\alpha -1}\|u(\sigma)-u(h+\sigma)\|\, dh\Big)^{p}   \, d\sigma
\\ & \leq C \big( [u]_{F^{s+\alpha}_{p,1}(\R,w_{\gamma};X)}^{(1)}\big)^p \leq C\|u\|_{F^{s+\alpha}_{p,1}(\R,w_\gamma;X)}^p.
\end{align*}
We therefore find
\begin{align*}
\|u(0)\|_{D_A(\theta,p)} \leq C\big( \|u \|_{F^{s}_{p,1}(\R,w_\gamma;D_A(\alpha,\infty))} + \|u\|_{F^{s+\alpha}_{p,1}(\R,w_\gamma;X)}\big)
\end{align*}
for all $u\in \Schw(\R;D_A(\alpha,\infty))$. By density, under the assumption \eqref{eq:assumptionStep1} of this step the trace operator $\tr_0$ extends continuously to a map
\begin{equation}\label{eq:trace_q=1}
\tr_0:F^{s+\alpha}_{p,1}(\R,w_\gamma;X) \cap F^{s}_{p,1}(\R,w_\gamma;D_A(\alpha,\infty))\to D_A(\theta,p).
\end{equation}

{\em Step 2.} Assume $-1<\gamma<p-1$ and $0< s< \frac{1+\gamma}{p} < s+\alpha$. If $s+\alpha<1$, then \eqref{eq:assumptionStep1}  holds and Step 1 applies. So assume $s+\alpha\geq 1$. Since $\frac{1+\gamma}{p}<1$, we can find $\sigma\in (0,1)$ such that $\frac{1+\gamma}{p} <s+\sigma\alpha<1$. Using that (see \eqref{reiteration} and \cite[Theorem 1.15.2]{Tr1})
$$(X,D_A(\alpha,\infty))_{1-\sigma,1} = D_A((1-\sigma)\alpha,1) \hookrightarrow D(A^{(1-\sigma)\alpha}),$$
and the mixed derivative embedding from Theorem \ref{thm:newton} (see also Remark \ref{rem:mixed-der}(i)), we get
\[F^{s+\alpha}_{p,1}(\R,w_\gamma;X) \cap F_{p,1}^s(\R,w_\gamma; D_A(\alpha,\infty)) \hookrightarrow F^{s+ \sigma\alpha}_{p,1}(\R,w_\gamma; D(A^{(1-\sigma)\alpha})).\]
Now let $\tilde X = D(A^{(1-\sigma)\alpha})$ and let $\tilde A$ be the realization of $A$ on $\tilde X$. Then $\tilde A$ is again a positive operator, and $D_{\tilde A}(\sigma\alpha,\infty) = D_A(\alpha,\infty)$ as a consequence of \eqref{Aiso}. For $\tilde \alpha = \sigma \alpha$ we thus have
\begin{align*}
F^{s+\alpha}_{p,1}(\R,w_\gamma;X) \cap F^{s}_{p,1}(\R,w_\gamma;D_A(\alpha,\infty))\hookrightarrow F^{s+\tilde \alpha}_{p,1}(\R,w_\gamma;\tilde X) \cap F^{s}_{p,1}(\R,w_\gamma;D_{\tilde A}(\tilde\alpha,\infty)).
\end{align*}
Since $\frac{1+\gamma}{p} < s+\tilde \alpha <1$, Step 1 applies and shows that $\tr_0$ extends continuously to a map
$$F^{s+\alpha}_{p,1}(\R,w_\gamma;X) \cap F^{s}_{p,1}(\R,w_\gamma;D_A(\alpha,\infty)) \to D_{\tilde A}\Big(s+\tilde \alpha - \frac{1+\gamma}{p},p\Big) = D_A(\theta,p),$$
using again \eqref{Aiso} for the last identity.

{\em Step 3.} Assume $-1<\gamma<p-1$ and $s\leq 0 < \frac{1+\gamma}{p} < s+\alpha$. Then we find $\sigma\in (0,1)$ such that $0 < s+ \sigma \alpha < \frac{1+\gamma}{p}$. Setting $\tilde s = s + \sigma \alpha$ and $\tilde \alpha = (1-\sigma)\alpha$, Theorem \ref{thm:newton} implies
\[F^{s+\alpha}_{p,1}(\R,w_\gamma;X) \cap F^{s}_{p,1}(\R,w_\gamma;D_A(\alpha,\infty)) \hookrightarrow   F^{\tilde s + \tilde \alpha}_{p,1}(\R,w_\gamma;X) \cap F^{\tilde s}_{p,1}(\R,w_\gamma;D_A({\tilde \alpha}, \infty)).\]
Since $0< \tilde s < \frac{1+\gamma}{p} < \tilde s + \tilde \alpha$ and $\theta = \tilde s +\tilde \alpha -\frac{1+\gamma}{p}$, we can apply Step 2 to obtain \eqref{eq:trace_q=1}.

{\em Step 4.} Assume $\gamma\geq p-1$, and that $s< \frac{1+\gamma}{p} < s+\alpha$ are arbitrary. Let $\tilde{s} = s-\frac{\gamma}{p}$. Since $\frac{\gamma}{p} > 0$ and $s - \frac{1+\gamma}{p} = \tilde s - \frac{1}{p}$, the Sobolev embedding from Theorem \ref{thm:SobB} with fixed integrability parameter $p$ gives
\[F^{s+\alpha}_{p,1}(\R,w_\gamma;X) \cap F^{s}_{p,1}(\R,w_\gamma;D_A(\alpha, \infty)) \hookrightarrow F^{\tilde{s}+\alpha}_{p,1}(\R;X) \cap F^{\tilde{s}}_{p,1}(\R;D_A(\alpha, \infty)).\]
Since $\tilde{s}<\frac{1}{p} < \tilde{s}+\alpha$ and $\theta = \tilde s + \alpha - \frac{1}{p}$, we obtain \eqref{eq:trace_q=1} from the Steps 2 and 3.

{\em Step 5.} We finally prove \eqref{trace-cont} by reducing to the case \eqref{eq:trace_q=1}, considered in the previous steps. As in Step 4 we use weighted Sobolev embeddings with fixed integrability $p$. For small $\varepsilon>0$ we set
$$\tilde{s} = s- \frac{\varepsilon}{p}, \qquad \tilde{\gamma} = \gamma-\varepsilon.$$
Since $\frac{\gamma}{p} > \frac{\tilde \gamma}{p}$ and $s - \frac{1+ \gamma}{p} = \tilde s- \frac{1+ \tilde \gamma}{p}$,  Theorem \ref{thm:SobB} implies
\begin{align*}
F^{s + \alpha}_{p,\infty}(\R,w_\gamma;X) \hookrightarrow  F^{\tilde{s} + \alpha}_{p,1}(\R,w_{\tilde \gamma};X), \qquad F^{s}_{p,\infty}(\R,w_\gamma;D_A(\alpha,\infty))  \hookrightarrow  F^{\tilde{s}}_{p,1}(\R,w_{\tilde \gamma};D_A(\alpha,\infty)).
\end{align*}
Combining these embeddings with \eqref{eq:trace_q=1} yields the continuity of
$$\tr_0:F^{s + \alpha}_{p,\infty}(\R,w_\gamma;X)\cap F^{s}_{p,\infty}(\R,w_\gamma;D_A(\alpha,\infty)) \to D_A\Big(\tilde s + \alpha - \frac{1+\tilde \gamma}{p},p\Big) = D_A(\theta,p),$$
and this shows \eqref{trace-cont}.
\end{proof}

To obtain a continuous right-inverse for $\tr_0$ we need extension operators for the half-line $\R_+ = (0,\infty)$. Set
$$W^{m,p}(\R_+,w_\gamma;X) := \{f\in L^p (\R_+,w_\gamma; X)\,:\, \partial_t^k f \in L^p(\R_+,w_\gamma;X) \text{ for every } k\leq m\},$$
where the derivatives are taken in the sense of distributions.
For a Banach space $X$ and $m\in \N_0$, a linear map $E_+:L_{1,\text{loc}}([0,\infty);X)\to L_{1,\text{loc}}(\R;X)$ with $(E_+ f)|_{\R_+} = f$ is called an {\em $m$-extension operator for $\R_+$} if it is bounded from $W^{l,p}(\R_+,w_{\gamma};X)$ to $W^{l,p}(\R,w_{\gamma};X)$ for all $p\in (1,\infty)$, $\gamma \in (-1,p-1)$ and $l\in \{0,\ldots,m\}$.

As in \cite[Theorem 5.19]{AF03}) we define $\textsf{E}_+^m$ and $\textsf{E}_+^{m,k}$  by
\begin{align}
(\textsf{E}^m_+ f)(t) &= \left\{
          \begin{array}{ll}
            f(t), & t> 0, \\
            \sum_{j=1}^{m+1} \lambda_j f(-j t), & t<0,
          \end{array}
        \right. \label{ext-1}
\\ (\textsf{E}_+^{m,k} f)(t) &= \left\{
          \begin{array}{ll}
            f(t), & t> 0, \\
            \sum_{j=1}^{m+1} (-j)^{k} \lambda_j f(-j t), & t<0,
          \end{array}
        \right.\label{ext-2}
\end{align}
where $(\lambda_j)_{j=1}^{m+1}$ is the unique solution of $\sum_{j=1}^{m+1} (-j)^l \lambda_j = 1$ for $l=0,...,m$.

The following extension of \cite[Theorem 5.19]{AF03} is straight forward.

\begin{lemma} \label{extension} \
For every Banach space $X$ and $m \in \N$ one has that $\emph{\textsf{E}}^m_+$ defines an $m$-extension operator for $\R_+$. Moreover, for $k\in \{1,\ldots,m\}$ one has that $\emph{\textsf{E}}_+^{m,k}$ defines an $(m-k)$-extension operator for $\R_+$, and
\[\partial_t^k \emph{\textsf{E}}^m_+  = \emph{\textsf{E}}_{+}^{m,k} \partial_t^k.\]
\end{lemma}
\begin{proof} It can be seen by a standard mollification and cut-off argument (using \cite[Theorem 2.1.4]{Tur00}) that $D = \{f|_{\R_+}\,:\; f\in C_c^\infty(\R;X)\}$ is dense in $W^{m,p}(\R_+,w_\gamma;X)$ for every $m\in \N_0$, $p\in (1,\infty)$ and $\gamma\in (-1,p-1)$. As in the proof of \cite[Theorem 5.19]{AF03} one can show that $\textsf{E}^m_+$ and $\textsf{E}_+^{m,k}$ extend continuously as required, and that the asserted identity holds true.
\end{proof}

For a positive operator $A$ on $X$ we set
$$R_{A}:\R_+\to \calL(X),\qquad R_{A}(t) = (1+t A)^{-1}.$$
Given $m,j\in \N$, for an $m$-extension operator $E_+$ for the half-line we define $\ext_{j,A}$ for $x\in X$ by
\begin{equation}\label{eq:extjA}
(\ext_{j,A}x)(t) =  (E_+ R_A^j x)(t),\qquad t \in \R.
\end{equation}
Then it is clear that $\ext_{j,A}$ defines a right-inverse for $\tr_0$, i.e., $\tr_0 \ext_{j,A}x = x$.
The operator $\ext_{j,A}$ depends on $m$ as well, but in all cases below it will be clear which extension $E_+$ is used.

We start with the following regularity result for $\ext_{j,A}$.

\begin{lemma}\label{RI}
Let $p\in (1,\infty)$, $\gamma\in (-1,p-1)$ and $\beta > \frac{1+\gamma}{p}$. For $m \geq \beta +1$, let $E_+$ be  an $m$-extension operator for $\R_+$ of the form \eqref{ext-1} or \eqref{ext-2}. Then for every $j\in \N$, the operator $\emph{\ext}_{j,A}$ defined by \eqref{eq:extjA} maps continuously
$$D_A\Big(\beta-\frac{1+\gamma}{p},p\Big) \to  F_{p,1}^{\theta \beta}(\R,w_\gamma; D(A^{(1-\theta)\beta})), \qquad \theta\in [0,1].$$
\end{lemma}

In the proof we will use Sobolev embeddings of weighted  $W^{k,p}$-spaces, $k\in \N_0$, into $F$-spaces. This is where we need an $A_p$-condition on the weight, which results in $\gamma < p-1$.

\begin{proof}
To prove the desired mapping properties of $\ext_{j,A}$ we distinguish between different values of $\theta\beta$ and $\beta$.

\emph{Step 1.} Assume there is  $k\in \N$ such that $k-1 + \frac{1+\gamma}{p} < \theta\beta < k$. Define $\tilde \gamma = \gamma+ p(k-\theta\beta)$. Then for $x\in D_A(\beta-\frac{1+\gamma}{p},p)$ we have
\begin{align*}
\|\ext_{j,A} x\|_{W^{k,p}(\R,w_{\tilde \gamma}; D(A^{(1-\theta)\beta}))} &\leq C\sum_{l=0}^k \|A^l R_A^{j+l}x\|_{L^p(\R_+,w_{\tilde \gamma}; D(A^{(1-\theta)\beta}))}.
\end{align*}
For fixed $0\leq l\leq k$ we estimate, setting $y = A^{(1-\theta)\beta}x$ and using $\|(\sigma+A)^{-j}\|_{\calL(X)} \leq \frac{C}{(1+\sigma)^j}$,
\begin{align}
\|A^{l} & R_A^{j+l} x\|_{L^p(\R_+,w_{\tilde \gamma}; D(A^{(1-\theta)\beta}))}^p \notag\\
& = \int_0^{\infty}  \|(t^{-1}+A)^{-j} (A(t^{-1}+A)^{-1})^ly\|^p t^{\tilde \gamma-(l+j)p+1} \, \frac{dt}{t} \notag\\
& = \int_0^{\infty}  \|(\sigma+A)^{-j} (A(\sigma+A)^{-1})^ly\|^p \sigma^{-\tilde \gamma+(l+j)p-1} \, \frac{d\sigma}{\sigma} \notag\\
& \leq C\|y\|^p \int_0^1\sigma^{-\tilde \gamma+(l+j)p-1} \, \frac{d\sigma}{\sigma}  + C \int_1^{\infty}   \sigma^{(l-\frac{1+\tilde \gamma}{p})p}\| (A(\sigma+A)^{-1})^ly\|^p  \, \frac{d\sigma}{\sigma}.\label{eq:m666}
\end{align}
Since $-\tilde \gamma+(l+j)p-1 > 0$ and $(1-\theta)\beta < \beta -\frac{1+\gamma}{p}$, the first summand can be estimated by $C\|y\|^p \leq C\|x\|_{D_A(\beta-\frac{1+\gamma}{p},p)}^p$. In case $l=0$ the second summand in \eqref{eq:m666} is again estimated by $C\|y\|^p$. For $1\leq l\leq k$
we will use the equivalent norm for $D_A(\eta,p)$, given by \eqref{eq:realinterResolvent}, with $\eta = l-\frac{1+\tilde \gamma}{p} = l-k + \theta\beta  -\frac{1+\gamma}{p}$. Hence the second summand in \eqref{eq:m666} can be estimated by
$$C\|y\|_{D_A(l-k+ \theta\beta  -\frac{1+\gamma}{p},p)}^p \leq C \|x\|_{D_A(\beta-\frac{1+\gamma}{p},p)}^p,$$
also using the mapping properties of $A^{(1-\theta)\beta}$ from \eqref{Aiso}. Therefore
\begin{align*}
\|\ext_{j,A} x\|_{W^{k,p}(\R,w_{\tilde \gamma}; D(A^{(1-\theta)\beta}))} &\leq C\|x\|_{D_A(\beta-\frac{1+\gamma}{p},p)}.\end{align*}
Using $\tilde \gamma \in (-1,p-1)$, $\tilde \gamma > \gamma$ and $k-\frac{1+\tilde \gamma}{p} = \theta\beta - \frac{1+\gamma}{p}$, the elementary embedding \eqref{eq:TriebelLizorkinW} and the Sobolev embedding from Theorem \ref{thm:SobB} imply that
\begin{align*}
W^{k,p}(\R,w_{\tilde \gamma}; D(A^{(1-\theta)\beta}))  \hookrightarrow F_{p,\infty}^{k}(\R,w_{\tilde \gamma};D(A^{(1-\theta)\beta}))\hookrightarrow F_{p,1}^{\theta\beta }(\R,w_\gamma;D(A^{(1-\theta)\beta})).
\end{align*}
Hence $\ext_{j,A}$ maps continuously as asserted if $k-1 + \frac{1+\gamma}{p} < \theta \beta < k$ for some $k\in \N$.

\emph{Step 2.} Assume  $0\leq \theta \beta \leq \frac{1+\gamma}{p}$. Arguing as in Step 1, we obtain
\begin{equation}\label{eq:m667}
\|\ext_{j,A} x\|_{F_{p,\infty}^0(\R,w_\gamma; D(A^\beta))} \leq C \|\ext_{j,A} x\|_{L^p(\R,w_\gamma;D(A^\beta))} \leq C \|x\|_{D_A(\beta-\frac{1+\gamma}{p},p)}.
\end{equation}
Choose $\e > 0$  with $\frac{1+\gamma}{p} + \e < \min\{\beta,1\}$. Using Step 1 with $\theta = \big(\frac{1+\gamma}{p}+\e\big)/\beta < 1$, we obtain that $\text{ext}_{j,A}$ maps $D_A(\beta-\frac{1+\gamma}{p},p)$ into $F_{p,1}^{\frac{1+\gamma}{p}+\e}(\R,w_\gamma; D(A^{\beta - \frac{1+\gamma}{p} -\e}))$. Together with \eqref{eq:m667} and an elementary embedding it maps into
\begin{equation*}
 F_{p,\infty}^{\frac{1+\gamma}{p}+\e}(\R,w_\gamma; D(A^{\beta - \frac{1+\gamma}{p} -\e})) \cap F_{p,\infty}^0(\R,w_\gamma; D(A^\beta)).
\end{equation*}
By Theorem \ref{thm:newton}, this intersection space embeds into $F_{p,\infty}^{\theta\beta}(\R,w_\gamma; D(A^{(1-\theta)\beta}))$, and therefore
\begin{equation}\label{eq:m668}
\ext_{j,A}:D_A\Big(\beta-\frac{1+\gamma}{p},p\Big)\to F_{p,\infty}^{\theta\beta}(\R,w_\gamma; D(A^{(1-\theta)\beta})).
\end{equation}
We improve this mapping property with a weighted Sobolev embedding as follows. For small $\e > 0$ we set $\tilde \beta  = \beta + \e/p$, $\tilde \theta = \frac{\theta\beta+\e/p}{\beta+\e/p}$ and $\tilde \gamma = \gamma + \e$. Then $\tilde \theta \tilde \beta \leq \frac{1+\tilde \gamma}{p}$. Thus \eqref{eq:m668} and Theorem \ref{thm:SobB} imply that $\text{ext}_{j,A}$ maps $D_A(\beta- \frac{1+\gamma}{p},p) = D_A(\tilde \beta- \frac{1+\tilde \gamma}{p},p)$ into
$$F_{p,\infty}^{\tilde \theta \tilde \beta}(\R, w_{\tilde \gamma}; D(A^{(1-\tilde \theta)\tilde \beta})) = F_{p,\infty}^{\theta \beta + \e/p}(\R, w_{\gamma+\e}; D(A^{(1-\theta)\beta}))    \hookrightarrow F_{p,1}^{\theta\beta}(\R,w_\gamma; D(A^{(1-\theta)\beta})).$$

\emph{Step 3.} It remains to deal with the case $k \leq \theta \beta \leq k + \frac{1+\gamma}{p}$ for some $k\in \N$. Let us first assume that there is $k_0\in \N$ such that $k_0 + \frac{1+\gamma}{p} < \beta < k_0$, the other case will be treated in the next step. Then there are $\theta_1,\theta_2  \in (0,1)$ such that $\theta_1 < \theta < \theta_2$ and
$$k -1 + \frac{1+\gamma}{p} < \theta_1 \beta < k, \qquad k+\frac{1+\gamma}{p} < \theta_2 \beta < k+1.$$ By Step 1, the operator $\text{ext}_{j,A}$ maps $D_A(\beta-\frac{1+\gamma}{p},p)$ into
\begin{equation*}
F_{p,1}^{\theta_2 \beta}(\R,w_\gamma;D(A^{(1-\theta_2)\beta})) \cap F_{p,1}^{\theta_1 \beta}(\R,w_\gamma;D(A^{(1-\theta_1)\beta})),
\end{equation*}
and Theorem \ref{thm:newton} implies that this space embeds into $F_{p,1}^{\theta \beta}(\R,w_\gamma;D(A^{(1-\theta)\beta}))$.

\emph{Step 4.} Finally, assume $k \leq \theta \beta \leq k + \frac{1+\gamma}{p}$ for some $k\in \N$ and $k_0 \leq \beta \leq  k_0 + \frac{1+\gamma}{p}$ for some $k_0\in \N$. Take $\tilde \beta \in (0,1)$ such that $k_0 + \frac{1+\gamma}{p} < \beta + \tilde \beta < k_0+1$. Then the Steps 1-3 apply to $\beta+\tilde \beta$ and we obtain
$$\ext_{j,A}: D_A\Big(\beta + \tilde \beta  -\frac{1+\gamma}{p},p\Big) \to F_{p,1}^{\tau(\beta +\tilde \beta)}(\R,w_\gamma; D(A^{(1-\tau)(\beta + \tilde \beta)})),\qquad \tau\in [0,1].$$
The choice $\tau = \frac{\theta\beta}{\beta+\tilde \beta}$ leads to $F_{p,1}^{\theta \beta }(\R,w_\gamma; D(A^{(1-\theta)\beta +\tilde \beta}))$ on the right-hand side. Applying $A^{-\tilde \beta }$ and using \eqref{Aiso} yields the continuity of $\ext_{j,A}$ as asserted.\end{proof}

This result can be used as follows to obtain a right-inverse for $\tr_0$ as required for \eqref{intro-Ftrace}.

\begin{proposition}\label{prop:extgeneral}
Let $A$ be a positive operator on $X$, let $p\in (1,\infty)$ and $\gamma \in (-1,p-1)$. Suppose that $s\in \R$ and $\alpha > 0$ satisfy $s < \frac{1+\gamma}{p} <  s+\alpha$. Then for $m \geq \max\{s+\alpha+1, \alpha+1\}$ there is an $m$-extension operator $E_+$ for $\R_+$, which is either of the form \eqref{ext-1} or \eqref{ext-2}, such that for any integer $j>\frac{\gamma+1}{p} -s$ the operator $\emph{\text{ext}}_{j,A}$ defined by  \eqref{eq:extjA} maps continuously
\begin{equation*}
D_A\Big(s+\alpha-\frac{1+\gamma}{p},p\Big) \to  F_{p,1}^{s+\alpha}(\R,w_\gamma; X) \cap F_{p,1}^{s}(\R,w_\gamma; D_A(\alpha,1)).
\end{equation*}
\end{proposition}

\begin{proof} \emph{Step 1.} In this step, let $E_+$ be an $m$-extension operator of the form \eqref{ext-1} or \eqref{ext-2}, and let $j\in \N$ be arbitrary. For any value of $s$, Lemma \ref{RI} with $\beta = s+\alpha$ and $\theta = 1$ yields
\begin{equation}\label{eq:m669}
\text{ext}_{j,A}:D_A\Big(s+\alpha-\frac{1+\gamma}{p},p\Big)\to F_{p,1}^{s+\alpha}(\R,w_\gamma; X).
\end{equation}
Next, for $s > \frac{1+\gamma}{p} -1$ we prove
\begin{equation}\label{eq:m671}
\text{ext}_{j,A}:D_A\Big(s+\alpha-\frac{1+\gamma}{p},p\Big)\to F_{p,1}^{s}(\R,w_\gamma; D(A^\alpha)).
\end{equation}
Here we distinguish two cases. For $s \geq  0$, \eqref{eq:m671} follows from Lemma \ref{RI}, applied with $\beta = s+\alpha$ and $\theta = \frac{s}{s+\alpha}$. Assume $\frac{1+\gamma}{p} -1< s< 0$.  Then $\tilde \gamma = \gamma-sp\in (-1,p-1)$. Since $s+\alpha-\frac{1+\gamma}{p} = \alpha-\frac{1+\tilde \gamma}{p}$, it follows from Lemma \ref{RI}, applied with $\beta = \alpha > \frac{1+\tilde \gamma}{p}$ and $\theta = 0$, that $\text{ext}_{j,A}$ maps into $F_{p,1}^{0}(\R,w_{\tilde \gamma}; D(A^{\alpha})).$ Since $\tilde \gamma > \gamma$, Theorem \ref{thm:SobB} yields that this space embeds into $F_{p,1}^{s}(\R,w_\gamma; D(A^{\alpha}))$, and \eqref{eq:m671} follows.

\emph{Step 2.} In case $s \leq \frac{1+\gamma}{p} -1$ we prove \eqref{eq:m671} for a special choice of $E_+$, which is as follows. We choose $k\in \N$ such that $\frac{1+\gamma}{p} -1<s +k \leq \frac{1+\gamma}{p} $. Take the $(m+k)$-extension operator $\textsf{E}_+^{m+k}$ from \eqref{ext-1} and the corresponding $m$-extension operator $\textsf{E}_+^{m+k,k}$ from \eqref{ext-2}, such that $\partial_t^k \textsf{E}_{+}^{m+k}  = \textsf{E}_{+}^{m+k,k}\partial_t^k$ by Lemma \ref{extension}. We define $\text{ext}_{j,A}$ as in \eqref{eq:extjA} with $E_+ = \textsf{E}_+^{m+k,k}$ and arbitrary $j\geq k+1 > \frac{1+\gamma}{p}-s$. This allows to estimate
\begin{align}
\|\text{ext}_{j,A} x\|_{F_{p,1}^{s}(\R,w_\gamma; D(A^{\alpha}))} & = C\|\textsf{E}_+^{m+k,k} (\partial_t^k R_A^{j-k} x)\|_{F_{p,1}^{s}(\R,w_\gamma; D(A^{\alpha-k}))} \nonumber\\
& = C\|\partial^k_t  (\textsf{E}_{+}^{m+k} R_A^{j-k}  x)\|_{F_{p,1}^{s}(\R,w_\gamma; D(A^{\alpha-k}))}\nonumber
\\ & \leq C\|\textsf{E}_{+}^{m+k} R_A^{j-k} x\|_{F_{p,1}^{s+k}(\R,w_\gamma; D(A^{\alpha-k}))}, \label{ext-bla}
\end{align}
using \cite[Proposition 3.10]{MeyVer1} in the last line. Now set $\tilde m = m+k$, $\tilde s = s+k$ and $\tilde \alpha = \alpha -k$, such that $\tilde m \geq \tilde s +\tilde \alpha +1$, $\tilde s + \tilde \alpha  > \frac{1+\gamma}{p}$ and $\tilde \alpha > 0$. Since $\tilde s >\frac{1+\gamma}{p} -1$, we may apply \eqref{eq:m671} with $\tilde m$, $\tilde s$ and $\tilde \alpha$ to $\textsf{E}_{+}^{m+k} R_A^{j-k}$, to the result
\begin{align*}
\|\textsf{E}_{+}^{m+k} R_A^{j-k} x\|_{F_{p,1}^{s+k}(\R,w_\gamma; D(A^{\alpha-k}))} = \|\textsf{E}_{+}^{\tilde m} R_A^{j-k}x\|_{F_{p,1}^{\tilde s }(\R,w_\gamma; D(A^{\tilde \alpha}))} \leq C\|x\|_{D_A(s+\alpha -\frac{1+\gamma}{p},p)}.
\end{align*}
Combining this with \eqref{ext-bla} gives \eqref{eq:m671} in case $s \leq \frac{1+\gamma}{p} -1$ for $\text{ext}_{j,A}$ as defined above. Together with Step 1, we conclude that $\text{ext}_{j,A}$ satisfies \eqref{eq:m669} and \eqref{eq:m671} for any $s$.

\emph{Step 3.}
It remains to improve the space $D(A^\alpha)$ on the right-hand side of \eqref{eq:m671} to $D_{A}(\alpha,1)$.
From \eqref{eq:m671} we see that $\ext_{j,A}$ maps $D_A(s+\alpha-\frac{1+\gamma}{p},p)$ into
$$F_{p,1}^{s-\varepsilon}(\R,w_{\gamma}; D(A^{\alpha+\varepsilon})),$$
where $\varepsilon>0$ is small.
Now, together with \eqref{eq:m669} we obtain that $\ext_{j,A}$ maps into
$$ F_{p,1}^{s+\alpha}(\R,w_{\gamma}; X) \cap F_{p,1}^{s-\e}(\R,w_{ \gamma}; D(A^{\alpha + \e})).$$
By Theorem \ref{thm:newton}, this intersection space embeds into $F_{p,1}^{s}(\R,w_{ \gamma}; D_A(\alpha,1)).$
\end{proof}

\begin{remark}\label{remark:universal}
In case $s>\frac{1+\gamma}{p}-1$, Lemma \ref{RI} and Step 1 of the proof above also apply to the total extension operator $\mathcal E_+$ for $\R_+$ from \cite[Theorem 5.21]{AF03} and $j=1$. Hence, in this case
$$(\ext_{A}x)(t) = \mathcal E_+ (1+tA)^{-1} x, \qquad t \in \R,$$ defines a continuous right-inverse for $\tr_0$ which only depends on $A$.
\end{remark}

Theorem \ref{thm:F-embed} and Proposition \ref{prop:extgeneral} yield the characterization \eqref{intro-Ftrace} as asserted in Theorem \ref{thm:main3}. The corresponding result  \eqref{intro-Btrace} for the $B$-spaces will be a consequence of the case $p=q$ and real interpolation.


\begin{proposition} \label{prop:B-case-inter} In the situation of Theorem \ref{thm:main3}, for all $q \in [1,\infty]$ we have
\begin{equation}\label{xxx1}
 \emph{\text{tr}}_0 \left (B^{s+\alpha}_{p,q}(\R,w_\gamma;X) \cap B^{s}_{p,q}(\R,w_\gamma;D_A(\alpha,r))\right) = D_A\Big(s+\alpha - \frac{1+\gamma}{p},q\Big).
\end{equation}
\end{proposition}
\begin{proof} \emph{Step 1.} We show the continuity of $\tr_0$. Applying the mixed derivative embedding from Theorem \ref{thm:newton} as in the Steps 2 and 3 of the proof of Theorem \ref{thm:F-embed}, we may assume that $\alpha < 1$. Further, Theorem \ref{thm:newton} implies that the intersection space in \eqref{xxx1} embeds into
$$\mathbb Y = B^{s+\alpha-\eta}_{p,q}(\R,w_\gamma;D_A(\eta,r)) \cap B^{s}_{p,q}(\R,w_\gamma;D_A(\alpha,r)),$$
where $\eta>0$ is  small. It thus suffices to prove that $\tr_0:\mathbb Y \to D_A(s+\alpha - \frac{1+\gamma}{p},q)$ is continuous.

From \eqref{intro-Ftrace} and $F_{p,p}^\sigma = B_{p,p}^\sigma$ we know that $\tr_0$ is continuous
\begin{equation}\label{xxx}
 \mathbb Y^{\pm\varepsilon}=B^{s\pm\varepsilon +\alpha-\eta}_{p,p}(\R,w_\gamma;D_A(\eta,r)) \cap B^{s\pm\varepsilon}_{p,p}(\R,w_\gamma;D_A(\alpha,r)) \to  D_A\Big(s\pm\varepsilon +\alpha - \frac{1+\gamma}{p},p\Big),
\end{equation}
where again $\varepsilon > 0$ is small. Applying the real interpolation functor $(\cdot, \cdot)_{1/2,q}$ to \eqref{xxx},  by reiteration we end up with $D_A(s+\alpha - \frac{1+\gamma}{p},q)$ on the right-hand side. It is more involved to interpolate the intersection spaces  $\mathbb Y^{\pm\varepsilon}$ on the left-hand side of \eqref{xxx}. As in \cite{MS11} we argue in an operator theoretic way.  For $\sigma \in \R$ and $u \in [1,\infty]$ we set
$$E_{p,u}^{\sigma} =  B^{\sigma}_{p,u}(\R,w_\gamma;D_A(\eta,r)).$$
On  $E_{p,u}^{\sigma}$ we consider the operator $\mathcal B = (1-\partial_t^2)^{(\alpha-\eta)/2}$ with domain
$D(\mathcal B) = E_{p,u}^{\sigma+\alpha-\eta}$. Using a weighted version of Mihlin's multiplier theorem in vector-valued Besov spaces (see \cite[Theorem 6.1]{Am97}), as in \cite[Example 10.2]{KuWe} one can show that $\mathcal B$ is an invertible sectorial operator with spectral angle equal to zero. This means that $\|\lambda(\lambda+\mathcal B)^{-1}\| \leq C$ on each sector in $\C$ with vertex in zero, see \cite[Section 2]{Haase:2}. Next, the pointwise realization of $A$ on $E_{p,u}^{\sigma}$ is again a positive operator, hence it is sectorial by a Neumann series argument. Since $\alpha < 1$, also $\mathcal A = A^{\alpha- \eta}$ is sectorial by \cite[Proposition 3.1.2]{Haase:2}. For the domain of $\mathcal A$ we have
$$D(\mathcal A) = B^{\sigma}_{p,u}(\R,w_\gamma;D_A(\alpha,r)).$$
It is clear that $\mathcal B$ and $\mathcal A$ commute in the resolvent sense. Noting that the $E_{p,u}^{\sigma}$ form a real interpolation scale (see \cite[Proposition 6.1]{MeyVer1}), it follows from a result of Da Prato and Grisvard (see \cite[Corollary 9.3.2, Theorem 9.3.5]{Haase:2}) that for all $\sigma$ and $u$ as above the operator sum $\mathcal  B+ \mathcal A$ with domain
$$D(\mathcal  B) \cap D(\mathcal A) =  B^{\sigma+\alpha-\eta}_{p,u}(\R,w_\gamma;D_A(\eta,r)) \cap B^{\sigma}_{p,u}(\R,w_\gamma;D_A(\alpha,r))$$
is a continuous isomorphism.

This can be used to show $(\mathbb Y^{-\varepsilon}, \mathbb Y^{\varepsilon})_{1/2,q} = \mathbb Y$ as follows. Choosing $\sigma = s\pm \varepsilon$ and $u = p$, we obtain that $\mathbb Y^{\pm\varepsilon}$ equals $D(\mathcal  B) \cap D(\mathcal A)$, where $\mathcal B$ and $\mathcal A$ are considered on $E_{p,p}^{s\pm\varepsilon}$ respectively. Hence $\mathcal  B+ \mathcal A$ is an isomorphism
\begin{equation}\label{xxx2}
 \mathcal  B+ \mathcal A:(\mathbb Y^{-\varepsilon}, \mathbb Y^{\varepsilon})_{1/2,q} \to (E_{p,p}^{s-\varepsilon},E_{p,p}^{s+\varepsilon})_{1/2,q} =  E_{p,q}^{s}.
\end{equation}
Now choosing $\sigma = s$ and $u = q$,  we obtain that $\mathbb Y = D(\mathcal  B) \cap D(\mathcal A)$ and that $(\mathcal  B+ \mathcal A)^{-1}: E_{p,q}^{s} \to \mathbb Y$ is an isomorphism. Combining this with \eqref{xxx2} yields $(\mathbb Y^{-\varepsilon}, \mathbb Y^{\varepsilon})_{1/2,q} = \mathbb Y.$ Therefore $\tr_0:\mathbb Y \to D_A(s+\alpha - \frac{1+\gamma}{p},q)$ is continuous, and the continuity of $\tr_0$ as asserted in \eqref{xxx1} follows.

\emph{Step 2.} From Proposition \ref{prop:extgeneral} and $F_{p,p}^\sigma = B_{p,p}^\sigma$ we know that there is a continuous right-inverse for $\tr_0$ mapping $D_A(s \pm\varepsilon +\alpha - \frac{1+\gamma}{p},p)$ into $B^{s\pm\varepsilon+\alpha}_{p,p}(\R,w_\gamma;X)$ and into $B^{s\pm\varepsilon }_{p,p}(\R,w_\gamma;D_A(\alpha,r))$, respectively. Applying $(\cdot, \cdot)_{1/2,q}$, we get that the right-inverse maps $D_A(s  +\alpha - \frac{1+\gamma}{p},q)$ continuously into $B^{s+\alpha}_{p,q}(\R,w_\gamma;X)$ and into $B^{s}_{p,q}(\R,w_\gamma;D_A(\alpha,r))$, respectively. This implies \eqref{xxx1}.
\end{proof}


We record a simple consequence of Theorem \ref{thm:main3}.

\begin{corollary}\label{cor:main}
Let $A$ be a positive operator on $X$, let $p\in (1,\infty)$, $q,r\in [1, \infty]$ and $\gamma \in(-1, p-1)$. Suppose $s\in \R$ and $\alpha > 0$ satisfy $s < \frac{1+\gamma}{p} <  s+\alpha$, and that $\beta > 0$ is such that $A^{\beta/\alpha}$ is a positive operator as well. Define $\theta = \left(s+\alpha - \frac{1+\gamma}{p}\right)\frac{\beta}{\alpha}.$
Then
\[\emph{\text{tr}}_0 \big(F^{s+\alpha}_{p,q}(\R,w_\gamma;X) \cap F^{s}_{p,q}(\R,w_\gamma;D_A(\beta,r))\big) = D_A(\theta,p),\]
\[\emph{\text{tr}}_0 \big(B^{s+\alpha}_{p,q}(\R,w_\gamma;X) \cap B^{s}_{p,q}(\R,w_\gamma;D_A(\beta,r))\big) = D_A(\theta,q).\]
\end{corollary}
\begin{proof} Since $D_A(\beta,r) = D_{A^{\beta/\alpha}}(\alpha,r)$ by reiteration, we can apply Theorem \ref{thm:main3} to $A^{\beta/\alpha}$, which yields that the trace spaces are $D_{A^{\beta/\alpha}}(s+\alpha- \frac{1+\gamma}{p},p) = D_A(\theta,p)$ and $D_A(\theta,q)$, respectively.
\end{proof}

Finally, as a byproduct of the above arguments we obtain a result on the regularity of orbits of analytic semigroups.

Denote by $D'(\R_+;X)$ the set of $X$-valued distributions on $\R_+$ (see \cite[Section III.1]{Ama95}). For $\mathcal A\in \{F,B\}$, as in \cite[Section 2.9]{Tri83} we set
\begin{equation}\label{def-restriction}
 \mathcal A_{p,q}^s(\R_+,w_\gamma;X) := \{f \in D'(\R_+;X)\,:\, \exists\,g\in \mathcal A_{p,q}^s(\R,w_\gamma;X) \text{ with } g|_{\R_+} = f\},
\end{equation}
which becomes a Banach space when equipped with the norm
$$\|f\|_{\mathcal A_{p,q}^s(\R_+,w_\gamma;X)} = \inf \{\|g\|_{\mathcal A_{p,q}^s(\R,w_\gamma;X)}\,:\, g|_{\R_+} = f\}.$$
Next, let the operator $-A$ generate an analytic $C_0$-semigroup $\{T(t)\}_{t\geq 0}$ on $X$. For simplicity we assume that the semigroup is exponentially stable, i.e., there are $M \geq 1$ and $\omega > 0$ such that $\|T(t)\|_{\calL(X)}\leq M e^{-\omega t}$ for all $t \geq 0$. Then by \cite[Theorem 1.14.5]{Tr1}, for $\alpha >0$ an equivalent norm for $D_A(\alpha,p)$ is given by
\begin{equation}\label{interpol-norm-analytic}
 y\mapsto  \Big(\int_0^\infty  t^{(m-\alpha) p} \|A^m T(t)y\|^p \,\frac{dt}{t}\Big)^{1/p},
\end{equation}
where $m > \alpha$ is an arbitrary integer.

A special case of the following result was obtained in \cite[Section 4.1]{DPZ08}.

\begin{theorem}\label{thm:semigcase}
Let $-A$ generate an exponentially stable analytic $C_0$-semigroup $\{T(t)\}_{t\geq 0}$ on $X$, let $p\in (1,\infty)$, $q\in [1,\infty]$ and $\gamma \in (-1,p-1)$. Suppose that $s\in \R$ and $\alpha > 0$ satisfy $s < \frac{1+\gamma}{p} <  s+\alpha$. Then $x\mapsto T(\cdot)x$ is a continuous map
\begin{equation*}
D_A\Big(s+\alpha-\frac{1+\gamma}{p},p\Big) \to  F_{p,1}^{s+\alpha}(\R_+,w_\gamma; X) \cap F_{p,1}^{s+\theta\alpha}(\R_+,w_\gamma; D_A((1-\theta)\alpha,1)),\qquad \theta\in [0,1),
\end{equation*}
\begin{equation*}
D_A\Big(s+\alpha-\frac{1+\gamma}{p},q\Big) \to  B_{p,q}^{s+\alpha}(\R_+,w_\gamma; X) \cap B_{p,q}^{s+\theta\alpha}(\R_+,w_\gamma; D_A((1-\theta)\alpha,1)),\qquad \theta\in [0,1).
\end{equation*}
\end{theorem}
\begin{proof} Using the norm \eqref{interpol-norm-analytic} instead of  \eqref{eq:realinterResolvent}, one may argue as in Lemma \ref{RI} and Proposition \ref{prop:extgeneral} to obtain that $x\mapsto E_+T(\cdot)x$ maps $D_A(s+\alpha-\frac{1+\gamma}{p},p)$ continuously into the $F$-intersection spaces, over $\R$ instead of $\R_+$. Here $E_+$ is a suitable extension operator for $\R_+$ as in Proposition \ref{prop:extgeneral}. Since the restriction to $\R_+$ is, by definition, continuous from spaces over $\R$ to those over $\R_+$, the assertion for the $F$-spaces follows. The assertion for the $B$-spaces follows from real interpolation, using \cite[Proposition 6.1]{MeyVer1}.
\end{proof}

\section{The  $L^p$-$L^q$ approach to the two-phase Stefan problem\\ with Gibbs-Thomson correction\label{sec:Stefan}}
As an application of the trace results we prove maximal $L^p$-$L^q$-regularity for the linearized, fully inhomogeneous two-phase Stefan problem with Gibbs-Thomson correction.

This parabolic initial-boundary value problem is posed on the two phases $\R_+^d= \{x\in \R^d: x_d > 0\}$ and $\R_-^d= \{x\in \R^d: x_d < 0\}$, being separated by a flat interface $\R^{d-1}$, which is the common boundary of the phases. The unknowns are the 'temperatures' $u_+$ and $u_-$ in the phases $\R_+^d$ and $\R_-^d$, and the 'height function' $h$, which only lives on the interface $\R^{d-1}$. Writing $\dot{\R}^d = \R^{d-1}\times  (\R\backslash \{0\})$ and $u$ for $u_\pm$ in the phases, the problem is given by
\begin{equation}\label{stefan-inhom}
\left \{
\begin{array}{rccl}
\mu u + \partial_t u -\Delta u & = &   f & \text{ in }\R_+\times \dot{\R}^d,   \\
\llbracket u \rrbracket & = & 0 & \text{ on }\R_+\times \R^{d-1},\\
u + \Delta' h & = & g_1 & \text{ on }\R_+\times \R^{d-1},  \\
\partial_t h -\llbracket \partial_{\nu} u \rrbracket   & = & g_2 & \text{ on }\R_+\times \R^{d-1}, \\
u|_{t=0}  & = & u_0 & \text{ in }\dot{\R}^d,   \\
h|_{t=0}  & = & h_0 & \text{ on }\R^{d-1}.
\end{array}
\right.
\end{equation}
Here $\mu > 0$ is a constant, and

$$\llbracket u \rrbracket = u_+|_{\R^{d-1}} - u_-|_{\R^{d-1}}, \qquad \llbracket \partial_\nu u \rrbracket = \partial_{x_d} u_+ |_{\R^{d-1}} - \partial_{x_d} u_- |_{\R^{d-1}},$$
denote the jump of $u$ and of the outer normal derivatives of $u$ along the common phase boundary $\partial \dot{\R}^d = \R^{d-1}$, respectively. In the third equation, $u$ denotes $u_+|_{\R^{d-1}} =  u_-|_{\R^{d-1}}$, where this identity follows from $\llbracket u \rrbracket = 0$. Moreover, $\Delta$ is the Laplacian on $\R^d$ and $\Delta'$ is the Laplacian on $\R^{d-1}$. The inhomogeneities $f,g_1, g_2$ and the initial values $u_0, h_0$ are assumed to be given.

One ends up with \eqref{stefan-inhom} after locally transforming and linearizing the full two-phase Stefan problem with Gibbs-Thomson correction, which is a free boundary problem modeling phase transitions in liquid-solid systems, to a fixed phase boundary and extending to $\dot{\R}^d$, see \cite[Section 7]{EPS03}. The graph of $h$ represents the transformed free phase boundary.

In a maximal $L^p$-$L^q$-regularity approach one looks for strong solutions $(u,h)$ which satisfy \eqref{stefan-inhom} pointwise almost everywhere. The corresponding one-phase problem was considered in an $L^p$-$L^p$-setting in \cite[Section 5]{DSS08}. The fully inhomogeneous two-phase problem \eqref{stefan-inhom} was treated in \cite[Theorem 6.1]{EPS03} in an $L^p$-$L^p$-setting. For trivial initial data $u_0 = 0$ and $h_0=0$ it was treated in an $L^p$-$L^q$-setting in \cite[Theorem 4.40]{DK14}, where $p\in (1,\infty)$ and $\frac{2p}{p+1} < q < 2p.$

The purpose of this section is to extend the maximal $L^p$-$L^q$-regularity result of \cite[Theorem 4.40]{DK14} for \eqref{stefan-inhom} to the case of nontrivial initial values $u_0$ and $h_0$.

In the $L^p$-$L^q$-approach Triebel-Lizorkin spaces naturally come into play for $p\neq q$ as the optimal time regularity of the boundary inhomogeneities and the unknown $h$. We also refer to \cite{DHP07, Wei02} for the case of a heat equation with inhomogeneous Dirichlet or Neumann boundary conditions. In general, the motivation to work in an $L^p$-$L^q$-setting  is when the scaling of the basic underlying space $L^p(\R_+; L^q(\dot{\R}^d))$ fits to the scaling of the problem under consideration only if $p\neq q$ (see e.g. \cite[Section 3]{Can04}, \cite{Giga86} or \cite[Section 1]{Shi11}).

In the $L^p$-$L^q$-approach one starts with
$$f \in \mathbb{E}_0= L^p(\R_+; L^q(\dot{\R}^{d})), \qquad p,q\in (1,\infty),$$
and looks for a solution $(u,h)$ such that, at first,
$$u\in \E_u = H^{1,p}(\R_+; L^q(\dot{\R}^d))\cap L^p(\R_+; H^{2,q}(\dot{\R}^d)).$$
The boundary inhomogeneities $g_1$ and $g_2$ should have at least the regularity of the terms involving $u$ arising in the corresponding equations. Combining \cite[Theorem 2.2]{JS08} with \cite[Proposition 3.23]{DK14} (see also \cite[Proposition 6.4]{DHP07}), this suggests that
$$g_1\in \mathbb{F}_1 = F_{p,q}^{1-1/(2q)}(\R_+; L^q(\R^{d-1})) \cap L^p(\R_+; B_{q,q}^{2-1/q}(\R^{d-1})),$$
$$g_2 \in \mathbb{F}_2 = F_{p,q}^{1/2-1/(2q)}(\R_+; L^q(\R^{d-1})) \cap L^p(\R_+; B_{q,q}^{1-1/q}(\R^{d-1})).$$
Here, the Triebel-Lizorkin spaces over the half-line are defined by restriction, see \eqref{def-restriction} and  \cite[Definition 3.4]{DK14}.

It is shown in \cite[Corollary 3.12]{DK14} that for all $s > 0$, $r\in \R$ and $\frac{2p}{p+1} < q < 2p$ the restriction from $\R$ to $\R_+$ is a retraction from $F_{p,q}^s(\R;H^{r,q}(\R^{d-1}))$ to $F_{p,q}^s(\R_+;H^{r,q}(\R^{d-1}))$, and that there exists a universal coretraction. Of course, the same is true for any vector-valued $L^p$-space. This allows to transfer results for $F$-spaces over $\R$ to the corresponding $F$-spaces over $\R_+$ by an extension-restriction argument. Here and below, the condition $\frac{2p}{p+1} < q < 2p$ imposed \cite{DK14} should not be essential, see also Remark \ref{stefan-rem-5}. Note that the condition in particular covers the case $p = q$.

The regularity of the boundary unknown $h$ should now be such that $\Delta' h \in \mathbb F_1$ and $\partial_t h \in \mathbb F_2$. We claim that
$$h \in \E_h = F_{p,q}^{3/2-1/(2q)}(\R_+; L^q(\R^{d-1}))\cap F_{p,q}^{1-1/(2q)}(\R_+; H^{2,q}(\R^{d-1})) \cap L^p(\R_+; B_{q,q}^{4-1/q}(\R^{d-1}))$$
is sufficient for this purpose. To apply the results from the previous sections we introduce the space
$$\tilde{\mathbb E}_h = F_{p,q}^{3/2-1/(2q)}(\R; L^q(\R^{d-1}))\cap F_{p,q}^{1-1/(2q)}(\R; H^{2,q}(\R^{d-1})) \cap L^p(\R; B_{q,q}^{4-1/q}(\R^{d-1})),$$
which is $\mathbb E_h$ with $\R_+$ replaced by $\R$, and further $\tilde{\mathbb F}_1$ and $\tilde{\mathbb F}_2$, which are $\mathbb F_1$ and $\mathbb F_2$ over $\R$, respectively.

\begin{lemma} \label{stefan-dt} For $p,q\in (1,\infty)$ the operators $\Delta': \tilde{\mathbb E}_h \to \tilde{\mathbb F}_1$ and $\partial_t:\tilde{\mathbb E}_h \to \tilde{\mathbb F}_2$ are continuous.
\end{lemma}
\begin{proof} The assertion for $\Delta'$ follows from a direct pointwise estimate in the $F$- and the $L^p$-norm. For the continuity of $\partial_t$, we use Theorem \ref{thm:newton} with $X_0= L^q$, $X_1 = H^{2,q}$, $\theta = 1/q + 2\e$ and $X_{1-\theta} = [X_0,X_1]_{1-\theta} =  H^{2-2/q-4\varepsilon,q}$. This gives
$$\tilde{\mathbb E}_h \hookrightarrow F_{p,q}^{3/2-1/(2q)}(\R; L^q(\R^{d-1}))\cap F_{p,q}^{1-1/(2q)}(\R; H^{2,q}(\R^{d-1})) \hookrightarrow F_{p,q}^{1+\varepsilon}(\R; H^{2-2/q-4\varepsilon,q}(\R^{d-1})),$$
where $\varepsilon > 0$ is such that $2-2/q-4\varepsilon > 1-1/q$. Using that then $F_{p,q}^{1+\varepsilon} \hookrightarrow H^{1,p}$ and $H^{2-2/q-4\varepsilon,q} \hookrightarrow B_{q,q}^{1-1/q}$, we get
$$\tilde{\mathbb E}_h \hookrightarrow F_{p,q}^{3/2-1/(2q)}(\R; L^q(\R^{d-1}))\cap H^{1,p}(\R; B_{q,q}^{1-1/q}(\R^{d-1})).$$
This embedding together with \cite[Proposition 3.10]{MeyVer1} shows that $\partial_t:\tilde{\mathbb E}_h \to \tilde{\mathbb F}_2$ is continuous.
\end{proof}

In the following we determine the temporal trace spaces of $\mathbb E_u$ and $\E_h$, to which $u_0 = u|_{t=0}$ and $h_0 = h|_{t=0} $ necessarily belong if $(u,h) \in \mathbb E_u \times \mathbb E_h$.

For $s > 0$, let $B_{q,p}^s(\dot{\R}^d)$ be the space of all $v_0 \in L^q(\R^d)$ such that $v_0 |_{\R_\pm^d} \in B_{q,p}^s(\R_\pm^d)$, where the latter spaces are as above defined by restriction (see \cite[Section 4.2.1]{Tr1}). It then follows from an extension-restriction argument and Theorem \ref{thm:main3} that the trace operator $\tr_0u = u|_{t=0}$ maps $\mathbb E_u$ continuously onto
$$X_u = (L^q(\dot{\R}^d), H^{2,q}(\dot{\R}^d))_{1-1/p,p} = B_{q,p}^{2-2/p}(\dot{\R}^d),$$
see \cite[Theorem 2.4.1]{Tr1} for the interpolation result. The temporal trace space of $\mathbb E_h$ will be deduced from the following.

\begin{proposition} \label{stefan-trace} Let $p,q\in (1,\infty)$ satisfy $1-1/(2q),1/2-1/(2q)\neq 1/p$, and consider again
$$\tilde{\mathbb E}_h = F_{p,q}^{3/2-1/(2q)}(\R; L^q(\R^{d-1}))\cap F_{p,q}^{1-1/(2q)}(\R; H^{2,q}(\R^{d-1})) \cap L^p(\R; B_{q,q}^{4-1/q}(\R^{d-1})).$$
Then $\emph{\tr}_0$ maps $\tilde{\mathbb E}_h$ continuously onto $X_h$, where
$$X_h  = B_{q,p}^{6-2/q - 4/p}(\R^{d-1})\quad \text{if }\; 1-1/(2q) < 1/p, \quad X_h  = B_{q,p}^{4- 1/q - 2/p}(\R^{d-1})\quad  \text{if } \;1-1/(2q) > 1/p.$$
Moreover, if $1/2-1/(2q) > 1/p$, then the operator $\emph{\tr}_0 \partial_t$ maps $\tilde{\mathbb E}_h$ continuously onto
$$X_{\partial_ t h}  = B_{q,p}^{2-2/q - 4/p}(\R^{d-1}).$$
In each case there is a continuous map $\mathcal R : X_h \times X_{\partial_t h} \to \tilde{\mathbb E}_h$ such that $\emph{\tr}_0 \mathcal R (h_0,h_1) = h_0$ and, in case $1/2-1/(2q) > 1/p$, such that also $\emph{\tr}_0 \partial_t \mathcal R (h_0,h_1) = h_1$ for all $(h_0,h_1)\in X_h\times X_{\partial_t h}$.
\end{proposition}
\begin{proof} To economize the notation we  write $L^q = L^q(\R^{d-1})$, $H^{s,q} = H^{s,q}(\R^{d-1})$ and so on. Throughout we consider an $m$-extension operator $E_+$ for the half-line from Lemma \ref{extension} with $m=3$, say. We will often use identities
$$(L^q, H^{k,q})_{\theta,p} = B_{q,p}^{\theta k},\qquad [L^q, H^{k,q}]_{\theta} = H^{\theta k,q}, \qquad \theta\in (0,1),\quad k\in \N.$$

\emph{Step 1.} Assume  $1-1/(2q) < 1/p$. Let $X = L^q$ and $A = (1-\Delta')^2$ with $D(A) = H^{4,q}$.  Then $H^{2,q} = D(A^{1/2}) \hookrightarrow D_A(1/2,\infty)$, which implies
$$\tilde{\mathbb E}_h \hookrightarrow F_{p,q}^{3/2-1/(2q)}(\R;X) \cap F_{p,q}^{1-1/(2q)}(\R; D_A(1/2, \infty)).$$
Thus by Theorem \ref{thm:main3}, $\tr_0$ maps $\tilde{\mathbb E}_h$ continuously into $D_A(3/2-1/(2q)-1/p,p) = B_{q,p}^{6-2/q - 4/p}.$

To obtain a right-inverse, we consider $\ext_{1,A}$ from Proposition \ref{prop:extgeneral} with respect to $A$. It follows that $\ext_{1,A}$ maps $D_A(3/2-1/(2q)-1/p,p)$ into
$$\mathbb Y = F_{p,1}^{3/2-1/(2q)}(\R;L^q) \cap F_{p,1}^{0}(\R; H^{6-2/q,q}),$$
also employing $D_A(3/2-1/(2q),1) = B_{q,1}^{6-2/q,q} \hookrightarrow H^{6-2/q,q},$ which is a consequence of \eqref{eq:TriebelLizorkinH} and \eqref{eq:monotony}.
Using $H^{6-2/q,q}\hookrightarrow B_{q,q}^{4-1/q}$ and the elementary embedding $F_{p,1}^{0}\hookrightarrow L^p$ from \eqref{eq:TriebelLizorkinH}, we obtain
$$\mathbb Y \hookrightarrow F_{p,1}^{0}(\R;B_{q,q}^{4-1/q})\hookrightarrow L^p(\R;B_{q,q}^{4-1/q}).$$
Further, it follows from Theorem \ref{thm:newton} with $\theta = \frac{1}{3-1/q}$ and $X_{\theta} = [L^q, H^{6-2/q,q}]_{\frac{1}{3-1/q}}  = H^{2,q}$ that
$$\mathbb Y \hookrightarrow F_{p,1}^{1-1/(2q)}(\R;H^{2,q}) \hookrightarrow F_{p,q}^{1-1/(2q)}(\R;H^{2,q}).$$
Altogether, we have shown that $\mathbb Y \hookrightarrow \tilde{\mathbb E}_h$, and thus $\mathcal R(h_0,h_1) = \ext_{1,A}h_0$ satisfies the requirements.

\emph{Step 2.} Assume $1-1/(2q) > 1/p$. We show the continuity of $\tr_0$. Consider $A = 1-\Delta'$ on $X = H^{2,q}$ with $D(A) = H^{4,q}$. Then $B_{q,q}^{4-1/q} = D_A(1-1/(2q),q)$, which together with \eqref{eq:TriebelLizorkinH} implies
\[\tilde{\mathbb E}_h\hookrightarrow F_{p,\infty}^{1-1/(2q)}(\R; X) \cap F^{0}_{p,\infty}(\R; D_A(1-1/(2q),q)).\]
Therefore, by Theorem \ref{thm:main3} we find that $\tr_0$ maps $\tilde{\mathbb E}_h$ continuously into
$$D_A(1-1/(2q)-1/p,p)  = B_{q,p}^{4-1/q-2/p}.$$

For a right-inverse we consider $A = 1-\Delta'$ on  $X= L^q$ with $D(A) = H^{2,q}$. Applying Proposition \ref{prop:extgeneral}, we get that $\mathcal R(h_0,h_1) = \ext_{1,A} h_0$ maps $B_{q,p}^{4-1/q-2/p} = D_{A}(2-1/(2q)-1/p,p)$
continuously into
\begin{equation*}
 \mathbb{Y} = F_{p,1}^{2-1/(2q)}(\R;X) \cap F_{p,1}^0(\R; D_{A}(2-1/(2q),1)).
\end{equation*}
The embeddings $D_{A}(2-1/(2q),1) = B^{4-1/q}_{q,1}\hookrightarrow B^{4-1/q}_{q,q}$ and \eqref{eq:TriebelLizorkinH} yield
$$\mathbb{Y} \hookrightarrow F_{p,q}^{3/2-1/(2q)}(\R;L^q) \cap L^p(\R; B^{4-1/q}_{q,q}).$$
Moreover,  Theorem \ref{thm:newton} applied with $\theta = \frac{1}{2-1/(2q)}$ and
$$X_{\theta} = (L^q, D_{A}(2-1/(2q),1))_{\frac{1}{2-1/(2q)},1} \hookrightarrow H^{2,q}$$ implies that $\mathbb Y \hookrightarrow F_{p,q}^{1-1/(2q)}(\R; H^{2,q})$. Hence $\mathbb Y\hookrightarrow \tilde{\mathbb E}_h$, showing that $\mathcal R$ maps as required.

\emph{Step 3.} Assume $ 1/2-1/(2q) > 1/p.$  We show the continuity of $\tr_0 \partial_t$.
Let $A = (1-\Delta')^2$ on $X = L^q$ with $D(A) = H^{4,q}$. Let $X_{\theta} = H^{4\theta,q}$ for $\theta\in [0,1]$.
Theorem \ref{thm:newton} gives
$$\tilde{\mathbb E}_h \hookrightarrow F_{p,q}^{3/2-1/(2q)}(\R; X) \cap F_{p,q}^{1-1/(2q)}(\R;X_{1/2}) \hookrightarrow F_{p,q}^1(\R; X_{1/2-1/(2q)}).$$
It thus follows from \cite[Proposition 3.10]{MeyVer1} that
$$\partial_t: \tilde{\mathbb E}_h \to F_{p,q}^{1/2-1/(2q)}(\R; X) \cap F_{p,q}^{0}(\R; X_{1/2-1/(2q)})$$
is continuous. Now Theorem \ref{thm:main3} implies that
$$\tr_0 \partial_t : \tilde{\mathbb E}_h  \to D_A(1/2-1/(2q)-1/p,p) = B_{q,p}^{2-2/q-4/p}$$
is continuous. For the right-inverse we let $X = L^q$, $A_0 = 1-\Delta'$ with $D(A_0) = H^{2,q}$ and $A_1 = (1-\Delta')^{2}$ with $D(A_1) = H^{4,q}$. Following the considerations in \cite[Section 4.1]{DPZ08}, we set
$$\mathcal R_0 = 2\ext_{1,A_0} - \ext_{1,2A_0}, \qquad \mathcal R_1 = \big( \ext_{1,A_1} - \ext_{1,2A_1}\big)A_1^{-1}.$$
Then $\tr_0 \mathcal R_0 = \text{id}$, $\tr_0 \partial_t \mathcal R_0 = 0$, $\tr_0 \mathcal R_1 = 0$ and $\tr_0 \partial_t \mathcal R_1 = \text{id}$. Hence
$$\mathcal R (h_0,h_1) = \mathcal R_0 h_0 + \mathcal R_1 h_1$$
satisfies $\tr_0 \mathcal R (h_0,h_1) = h_0$ and $\tr_0 \partial_t \mathcal R (h_0,h_1) = h_1$. In Step 2 we have shown that $\mathcal R_0 : X_h \to \tilde{\mathbb E}_h$ is continuous. Moreover, since $A_1^{-1}$ maps $X_{\partial_t h}$ to $B_{q,p}^{6-2/q-4/p}$ and since we have shown in Step 1 that $\ext_{1,A_1}$ maps $B_{q,p}^{6-2/q-4/p}$ to $\tilde{\mathbb E}_h$, we obtain that $\mathcal R_1 : X_{\partial_t h} \to \tilde{\mathbb E}_h$ is continuous. Therefore $\mathcal R: X_h \times X_{\partial_t h} \to \tilde{\mathbb E}_h$ is continuous as well.
\end{proof}

\begin{remark} \label{stefan-rem-1} Employing \cite[Corollary 3.12]{DK14} and an extension-restriction argument, for $\frac{2p}{p+1} < q < 2p$ the assertions of Lemma \ref{stefan-dt} remain true if one replaces $\tilde{\mathbb E}_h$, $\tilde{\mathbb F}_1$, $\tilde{\mathbb F}_2$ by $\mathbb E_h$, $\mathbb F_1$, $\mathbb F_2$. Similarly, for $\frac{2p}{p+1} < q < 2p$ the assertions of Proposition \ref{stefan-trace} remain true if one replaces $\tilde{\mathbb E}_h$ by $\mathbb E_h$.
\end{remark}

\begin{remark}\label{remark-DPZ} The methods from the proof above are not restricted to the special form of $\tilde{\mathbb E}_h$ and apply to general intersection spaces that arise in the context of the $L^p$-$L^q$-approach to initial-boundary value problems with inhomogeneous symbols as considered in \cite{DK14, DPZ08}. We can further allow temporal weights as in \cite{MS12} in order to obtain flexibility for the initial regularity. The weighted approach is very useful when studying the long-time behavior of solutions (see e.g. \cite{PSSS12}). For instance, the arguments from the proof above allow to determine the precise temporal trace space of
\begin{align*}
F_{p,q}^{1+\kappa_0}(\R,w_\gamma; L^q(\R^{d-1}))&  \cap  L^p(\R,w_\gamma; B_{q,q}^{l + 2m\kappa_0}(\R^{d-1}))\\
&  \cap H^{1,p}(\R,w_\gamma; B_{q,q}^{2m\kappa_0}(\R^{d-1})) \cap \bigcap_{j=0}^m F_{p,q}^{\kappa_j}(\R,w_\gamma; H^{k_j,q}(\R^{d-1})),
\end{align*}
where $w_{\gamma}(t) = |t|^\gamma$ for $\gamma \in (-1,p-1)$. This is the space of the boundary unknown in the $L^p$-$L^q$-approach to boundary value problems of relaxation type (see \cite[Section 2]{DPZ08} and \cite{MS12}). Here the numbers $\kappa_j$ and $k_j$ are determined by the orders of the differential operators involved in the problem.
\end{remark}

If temporal traces exist, then compatibility conditions for the data are required to obtain a strong solution $(u,h)\in \mathbb E_u \times \mathbb E_h$ of \eqref{stefan-inhom}. If $1-1/(2q) > 1/p$, then the static boundary conditions imply that $\llbracket u_0\rrbracket = 0$ and $g_1|_{t=0} = u_0 + \Delta'h_0$ necessarily hold as well. In case $1/2-1/(2q) > 1/p$ also $\partial_t h|_{t=0}$ exists, and thus the dynamic boundary condition implies that $g_2|_{t=0} + \llbracket \partial_\nu u_0 \rrbracket$ enjoys at least the same regularity as $\partial_t h |_{t=0}$, i.e., that $ g_2|_{t=0} + \llbracket \partial_\nu u_0 \rrbracket \in X_{\partial_t h} = B_{q,p}^{2-2/q - 4/p}(\R^{d-1}).$ Observe that this condition is not trivial: if $g_2$ is an arbitrary element of $\mathbb F_2$, then the trace theorem only gives $g_2|_{t=0}\in B_{q,p}^{1-1/q-2/p}(\R^{d-1})$, which is only half of the required smoothness.

After these considerations we can extend \cite[Theorem 4.40]{DK14} to nontrivial initial values and  show maximal $L^p$-$L^q$-regularity for the fully inhomogeneous problem \eqref{stefan-inhom}. For the convenience of the reader we recall the spaces
$$\mathbb{E}_0 = L^p(\R_+; L^q(\dot{\R}^{d})),  \qquad \E_u = H^{1,p}(\R_+; L^q(\dot{\R}^d))\cap L^p(\R_+; H^{2,q}(\dot{\R}^d)),\quad X_u = B_{q,p}^{2-2/p}(\dot{\R}^d),$$
$$\mathbb{F}_1 = F_{p,q}^{1-1/(2q)}(\R_+; L^q(\R^{d-1})) \cap L^p(\R_+; B_{q,q}^{2-1/q}(\R^{d-1})),$$
$$ \mathbb{F}_2 = F_{p,q}^{1/2-1/(2q)}(\R_+; L^q(\R^{d-1})) \cap L^p(\R_+; B_{q,q}^{1-1/q}(\R^{d-1})),$$
$$\E_h = F_{p,q}^{3/2-1/(2q)}(\R_+; L^q(\R^{d-1}))\cap F_{p,q}^{1-1/(2q)}(\R_+; H^{2,q}(\R^{d-1})) \cap L^p(\R_+; B_{q,q}^{4-1/q}(\R^{d-1})),$$
and further the trace spaces $X_h$ and $X_{\partial_t h}$ determined in Proposition \ref{stefan-trace}.

\begin{theorem} \label{stefan-thm} Let $p,q\in (1,\infty)$ such that  $\frac{2p}{p+1} < q < 2p$ and $1-1/(2q), 1/2-1/(2q) \neq  1/p$, and let $\mu > 0$. Then there is a unique strong solution $(u,h)\in \E_u \times \E_h$ of \eqref{stefan-inhom} if and only if
$$f\in \mathbb{E}_0,\qquad g_1\in \mathbb{F}_1, \qquad g_2\in \mathbb{F}_2, \qquad u_0 \in X_u, \qquad h_0 \in  X_h,$$
and the compatibility conditions
$$\llbracket u_0 \rrbracket = 0, \quad g_1|_{t=0} = u_0 - \Delta' h_0  \qquad \text{ if }\,1-1/(2q) >  1/p,$$
$$ g_2|_{t=0} + \llbracket \partial_\nu u_0\rrbracket \in X_{\partial_t h} \quad \text{ if }\,1/2-1/(2q) >  1/p,$$
are satisfied. There is a constant $C>0$, which is independent of the data, such that
$$\|u\|_{\E_u} + \|h\|_{\E_h} \leq C\big( \|f\|_{\mathbb{E}_0} + \|g_1\|_{\mathbb{F}_1} + \|g_2\|_{\mathbb{F}_2} + \|u_0\|_{X_u} + \|h_0\|_{X_h}\big).$$
\end{theorem}

\begin{remark} \ \label{stefan-rem-5}
\begin{enumerate}[(i)]{}
 \item In case $p= q$ we precisely recover the result of \cite[Theorem 6.1]{EPS03}.
 \item Theorem \ref{stefan-thm} is the basic ingredient to treat the original two-phase Stefan problem with Gibbs-Thomson correction \cite[Problem (1.3)]{EPS03} in an $L^p$-$L^q$-setting.
 \item The restriction $\frac{2p}{p+1} < q < 2p$ is due to the results in \cite{DK14} and, as indicated there, should not be essential. If one  removes this condition in the results of \cite{DK14}, then the combination with our trace results from Proposition \ref{stefan-trace} gives maximal $L^p$-$L^q$ regularity for all $p,q\in (1,\infty)$, with the same arguments as given in the sequel.
\end{enumerate}
\end{remark}

\begin{proof}[Proof of Theorem \ref{stefan-thm}.] \emph{Step 1.} The necessity of the regularity of the data and the compatibility conditions are a consequence of our previous considerations. Further, uniqueness of solutions follows from the homogeneous case \cite[Theorem 4.40]{DK14}. In the following we show the existence of a solution for given data.

\emph{Step 2.} We claim that for all $f\in \mathbb E_0$ and $u_0 \in X_u$ with $\llbracket u_0 \rrbracket = 0$ if $1-1/p > 1/(2q)$ there is $\widetilde u \in \mathbb E_u$ satisfying
\begin{equation}\label{stefan-1}
 \mu \widetilde u + \partial_t \widetilde u - \Delta \widetilde u = f \quad \text{in }\R_+ \times \dot{\R}^d, \qquad \llbracket \widetilde u\rrbracket = 0\quad \text{on }\R_+\times \R^{d-1}, \qquad \widetilde u|_{t=0} = u_0 \quad \text{ in } \dot{\R}^d.
\end{equation}
To see this, extend $u_0|_{\R_-^d}$ to $\widetilde{u}_0 \in B_{q,p}^{2-2/p}(\R^d)$, and use \cite[Proposition 6.1]{DHP07} to define $w_- \in H^{1,p}(\R_+; L^q(\R_-^d))\cap L^p(\R_+; H^{2,q}(\R_-^d))$ as the restriction to $\R_-^d$ of the unique solution of
$$\mu v + \partial_t v -  \Delta v = f\quad \text{in }\R_+ \times \R^d, \qquad v|_{t=0} = \widetilde{u}_0 \quad \text{ in } \R^d.$$
Next, setting $u_{0,\pm} = u_0|_{\R_\pm^d}$, one has
$$u_{0,+}|_{\R^{d-1}} = u_{0,-}|_{\R^{d-1}} = w_-|_{\R^{d-1}, t=0}$$
if $1-1/p> 1/(2q)$, as a consequence of $\llbracket u_0 \rrbracket = 0$. We may thus use \cite[Proposition 6.4]{DHP07}  to define $w_+ \in H^{1,p}(\R_+; L^q(\R_+^d))\cap L^p(\R_+; H^{2,q}(\R_+^d))$ as the unique solution of
$$\mu v + \partial_t v -  \Delta v = f|_{\R_+^d} \quad \text{in }\R_+ \times \R_+^d, \quad  v|_{\R^{d-1}} = w_-|_{\R^{d-1}} \quad \text{on }\R_+\times \R^{d-1}, \quad v|_{t=0} = u_{0,+} \quad \text{ in } \R_+^d.$$
Now the function $\widetilde u\in \E_u$, defined by $\widetilde u|_{\R_\pm^d} = w_\pm$, satisfies \eqref{stefan-1}.

\emph{Step 3.} By Proposition \ref{stefan-trace} and Remark \ref{stefan-rem-1} there is $\widetilde{h} \in \E_h$ such that $\widetilde{h}|_{t=0} = h_0$ and $\partial_ t \widetilde{h}|_{t=0} = g_2|_{t=0} + \llbracket \partial_\nu u_0\rrbracket$ if $1/2-1/(2q) > 1/p$. For  $\widetilde h$ and the solution $\widetilde u$ of \eqref{stefan-1} we consider the problem
\begin{equation}\label{stefan-3}
\left \{
\begin{array}{rcll}
\mu w + \partial_t w -\Delta w & = &   0 & \text{ in }\R_+\times \dot{\R}^d,   \\
\llbracket w \rrbracket & = & 0 & \text{ on }\R_+\times \R^{d-1},\\
w + \Delta' \sigma  & = & g_1 - \widetilde{u} - \Delta' \widetilde{h}& \text{ on }\R_+\times \R^{d-1},  \\
\partial_t \sigma - \llbracket \partial_\nu w \rrbracket   & = & g_2 - \partial_t \widetilde{h} + \llbracket \partial_\nu \widetilde u\rrbracket & \text{ on }\R_+\times \R^{d-1}, \\
w|_{t=0}  & = & 0 & \text{ in }\dot{\R}^d,   \\
\sigma|_{t=0}  & = & 0 & \text{ on }\R^{d-1}.
\end{array}
\right.
\end{equation}
Since $\widetilde{u}|_{\R^{d-1}} + \Delta' \widetilde{h} \in \mathbb F_1$ by \cite[Proposition 6.4]{DHP07} and Lemma \ref{stefan-dt}, and further, by the compatibility assumption
$$g_1|_{t=0} - \widetilde{u}|_{\R^{d-1},t=0} - \Delta' \widetilde{h}|_{t=0} = 0\qquad \text{if }1-1/(2q) >  1/p,$$
it follows from \cite[Proposition 3.14]{DK14} that
$$g_1 - \widetilde{u}|_{\R^{d-1}} - \Delta' \widetilde{h} \in   {}_0F_{p,q}^{1-1/(2q)}(\R_+; L^q(\R^{d-1}))\cap L^p(\R_+; B_{q,q}^{2-1/p}(\R^{d-1})),$$
where ${}_0F_{p,q}^{\beta}$ denotes vanishing traces at $t=0$. We further have $\partial_t \widetilde{h} - \llbracket \partial_\nu \widetilde u\rrbracket \in \mathbb F_2$ by Remark \ref{stefan-rem-1} and  \cite[Proposition 6.4]{DHP07}, and as before it follows from $g_2|_{t=0} -\partial_ t \widetilde{h}|_{t=0} + \llbracket \partial_\nu u\rrbracket|_{t=0} = 0 $ if $1/2-1/(2q) > 1/p$ that
$$g_2 - \partial_t \widetilde{h} + \llbracket \partial_\nu \widetilde u \rrbracket \in {}_0F_{p,q}^{1/2-1/2p}(\R_+; L^q(\R^{d-1}))\cap L^p(\R_+; B_{q,q}^{1-1/p}(\R^{d-1})).$$
Hence \cite[Theorem 4.40]{DK14} provides a solution $(w,\sigma)\in \mathbb E_u\times \mathbb E_h$ of \eqref{stefan-3}. Now $(w+\widetilde u, \sigma + \widetilde h) \in \mathbb E_u\times \mathbb E_h$ is the solution of \eqref{stefan-inhom}. The asserted estimate of $(u,h)$ in terms of the data follows from the estimates in \cite{DHP07, DK14} for partially inhomogeneous problems as above and the continuity of the extension operator from Proposition \ref{stefan-trace}.
\end{proof}

\def\polhk#1{\setbox0=\hbox{#1}{\ooalign{\hidewidth
  \lower1.5ex\hbox{`}\hidewidth\crcr\unhbox0}}} \def\cprime{$'$}
  \def\cprime{$'$} \def\cprime{$'$}

\end{document}